\documentclass[a4paper,12pt,reqno,english]{amsart}
\usepackage[utf8]{inputenc}
\usepackage[T1]{fontenc}
\usepackage{babel}

\usepackage{amssymb}

\usepackage[bookmarks = false, colorlinks, citecolor = blue, urlcolor = blue]{hyperref}
\usepackage[titletoc,toc,title]{appendix}
\usepackage[capitalise, nameinlink]{cleveref}
\usepackage{tikz}

\numberwithin{equation}{section}
\numberwithin{figure}{section}

\theoremstyle{plain}
\newtheorem{theorem}{Theorem}[section]

\theoremstyle{plain}
\newtheorem*{theorem*}{Theorem}

\theoremstyle{plain}
\newtheorem{proposition}[theorem]{Proposition}

\theoremstyle{plain}
\newtheorem{lemma}[theorem]{Lemma}

\theoremstyle{plain}
\newtheorem{corollary}[theorem]{Corollary}

\theoremstyle{definition}
\newtheorem{definition}[theorem]{Definition}

\theoremstyle{definition}
\newtheorem{notation}[theorem]{Notation}

\theoremstyle{definition}

\theoremstyle{remark}
\newtheorem{remark}[theorem]{Remark}

\newcommand{\extm}[1]{
\mathfrak{e}_{#1}
}

\newcommand{\intm}[1]{
\mathfrak{i}_{#1}
}

\newcommand{\Uql}{U_q(\mathfrak{l})}
\newcommand{\Lambdaq}{\Lambda_q(\mathfrak{u}_+)}
\newcommand{\Lambdak}{\Lambda^k_q(\mathfrak{u}_+)}

\begin{document}

\title[On the Dolbeault-Dirac operators on quantum projective spaces]{On the Dolbeault-Dirac operators on quantum projective spaces}

\author{Marco Matassa}%

\email{marco.matassa@gmail.com, marco.matassa@math.univ-bpclermont.fr}

\address{Université Clermont Auvergne, CNRS, LMBP, F-63000 Clermont-Ferrand, France.}

\begin{abstract}
We consider Dolbeault–Dirac operators on quantum projective spaces, following Krähmer and Tucker-Simmons.
The main result is an explicit formula for their squares, up to terms in the quantized Levi factor, which can be expressed in terms of some central elements.
This computation is completely algebraic.
These operators can also be made to act on appropriate Hilbert spaces. Using the formula mentioned above, we easily find that they have compact resolvent, thus obtaining a result similar to that of D'Andrea and Dąbrowski.
\end{abstract}

\maketitle

\section{Introduction}

The aim of this paper is to provide a detailed study of a class of Dolbeault–Dirac operators on quantized symmetric spaces introduced in \cite{qflag2}.
We will consider the case of quantum projective spaces and give a detailed computation of the square of these operators.
The main result is that they can be expressed essentially in terms of certain central elements, which in the classical limit reduce to the usual quadratic Casimirs.
This can be seen as a quantum analogue of the classical Parthasarathy formula \cite{partha}.
As a corollary of this formula, we easily obtain a result similar to the one appearing in \cite{dd-proj}, namely that these operators have compact resolvent when acting on the appropriate Hilbert spaces.

First we give a brief survey of some of the relevant literature.
Important results on quantized irreducible generalized flag manifolds, of which projective spaces are examples, were obtained by Heckenberger and Kolb in \cite{flagcalc1, flagcalc2}: they show that these spaces admit a canonical $q$-analogue of the de Rham complex, with the homogenous components having the same dimensions as in the classical case.
The first-order differential calculi coming from this construction can be implemented by commutators with Dirac operators, in the sense of spectral triples, as shown by Krähmer in \cite{qflag}.
However, due to the implicit nature of this construction, he could not prove that these operators have compact resolvent in general. Recall that this is an important requirement for a spectral triple.
This difficulty motivated the subsequent construction of Krähmer and Tucker-Simmons in \cite{qflag2}: they connect this approach with the theory of the braided symmetric and exterior algebras of Berenstein and Zwicknagl \cite{bezw}, as a first step towards proving a quantum analogue of the Parthasarathy formula.
For the case of quantum projective spaces, Dirac operators defining spectral triples were introduced by D'Andrea and Dąbrowski in \cite{dd-proj}, generalizing previous results for the Podleś sphere \cite{ds-pod, twisted-podles} and for the quantum projective plane \cite{ddl-plane}.

In this paper we therefore revisit the construction of \cite{dd-proj}, connecting it with the general framework developed in \cite{qflag2}.
There are important differences, however.
First of all we avoid some seemingly ad-hoc choices made in \cite{dd-proj}, by working within a general scheme. Secondly, and most importantly, our results are more general. This is because of the algebraic setting we adopt, which defines the Dolbeault–Dirac operator $D$ as an element of $U_{q}(\mathfrak{g}) \otimes \mathrm{Cl}_{q}$. Since the operator does not act on a Hilbert space from the outset, the cancellations in the computation of $D^2$ take place at this more abstract level.

To obtain these results, it will be necessary to make some minor modifications to the setup of \cite{qflag2}.
This is because of some normalizations which are fixed in the cited paper, which create certain issues even in the classical limit.
Throughout this paper we will work in full generality in this respect.
As we will see at the end, certain choices will be singled out by the requirement that $D^{2}$ should take a simple form, as in the classical case.

The structure of the paper is as follows. In \cref{sec:notation} we fix our notations and conventions.
In \cref{sec:classical} we review the classical picture of Dolbeault–Dirac operators and the computation of their squares.
In \cref{sec:algebras} we introduce the relevant quantum symmetric and exterior algebras.
In \cref{sec:pairing} we discuss the choices for the dual pairings and Hermitian inner products.
In \cref{sec:clifford} we introduce the quantum Clifford algebra and give commutation relations for its generators.
In \cref{sec:root} we describe the quantum root vectors corresponding to the radical roots, together with their commutation relations.
In \cref{sec:dolbeault} we prove our main result on the squares of Dolbeault–Dirac operators on quantum projective spaces.
Finally in \cref{sec:casimir} we give some results related to the central elements appearing in the computation.

\section{Notations and conventions}
\label{sec:notation}

In this section we fix most of our notations and conventions and briefly review some facts about simple Lie algebras, parabolic subalgebras and quantized enveloping algebras.

\subsection{Parabolic subalgebras}

Let $\mathfrak{g}$ be a finite-dimensional complex simple Lie algebra with a fixed Cartan subalgebra $\mathfrak{h}$.
We denote by $\Delta(\mathfrak{g})$ the root system, by $\Delta^{+}(\mathfrak{g})$ a choice of positive roots and by $\Pi = \{ \alpha_{1}, \cdots, \alpha_{r} \}$ the simple roots.
Denote by $(\cdot, \cdot)$ the symmetric bilinear form on $\mathfrak{h}^{*}$ induced by the Killing form.
In the simply-laced case we have $(\alpha_{i}, \alpha_{j}) = a_{ij}$, where $a_{i j}$ are the entries of the Cartan matrix, and moreover $a_{i j} \in \{-1, 0\}$ for $i \neq j$.

For parabolic subalgebras we will follow the presentation given in \cite[Section 2.2]{qflag2}.
Let $S \subset \Pi$ be a subset of the simple roots.
Then we set
$$
\Delta(\mathfrak{l}) = \mathrm{span}(S) \cap \Delta(\mathfrak{g}), \quad
\Delta(\mathfrak{u}_{+}) = \Delta^{+}(\mathfrak{g}) \backslash \Delta^{+}(\mathfrak{l}).
$$
In terms of these roots we define
$$
\mathfrak{l} = \mathfrak{h} \oplus \bigoplus_{\alpha \in \Delta(\mathfrak{l})} \mathfrak{g}_{\alpha}, \quad
\mathfrak{u}_{\pm} = \bigoplus_{\alpha \in \Delta(\mathfrak{u}_{+})} \mathfrak{g}_{\pm \alpha}, \quad
\mathfrak{p} = \mathfrak{l} \oplus \mathfrak{u}_{+}.
$$
It follows from the definitions that $\mathfrak{l}$ and $\mathfrak{u}_{\pm}$ are Lie subalgebras of $\mathfrak{g}$.
We call $\mathfrak{p}$ the \emph{standard parabolic subalgebra} associated to $S$ (we will omit the dependence on $S$ in the following). The subalgebra $\mathfrak{l}$ is reductive and is called the \emph{Levi factor} of $\mathfrak{p}$, while $\mathfrak{u}_{+}$ is a nilpotent ideal of $\mathfrak{p}$ called the \emph{nilradical}. We refer to the roots of $\Delta(\mathfrak{u}_{+})$ as the \emph{radical roots}.

We will consider \emph{cominuscule} parabolics: these have the property that all radical roots contain a certain simple root $\alpha_t$ with multiplicity $1$.
For these parabolic subalgebras we have the commutation relations $[\mathfrak{u}_{+}, \mathfrak{u}_{-}] \subset \mathfrak{l}$.
This follows from the general commutation relations $[E_\alpha, E_\beta] = c_{\alpha, \beta} E_{\alpha + \beta}$, together with the fact that $\alpha_t$ appears with multiplicity $1$.

The adjoint action of $\mathfrak{p}$ on $\mathfrak{g}$ descends to an action on $\mathfrak{g} / \mathfrak{p}$. The decomposition $\mathfrak{g} = \mathfrak{u}_{-} \oplus \mathfrak{p}$ gives $\mathfrak{g} / \mathfrak{p} \cong \mathfrak{u}_{-}$ as $\mathfrak{l}$-modules.
With respect to the Killing form of $\mathfrak{g}$ both $\mathfrak{u}_{+}$ and $\mathfrak{u}_{-}$ are isotropic and we have $\mathfrak{l} = \mathfrak{u}^{\perp}$, where $\mathfrak{u} = \mathfrak{u}_{+} \oplus \mathfrak{u}_{-}$.
The pairing $\mathfrak{u}_{-} \otimes \mathfrak{u}_{+} \to \mathbb{C}$ coming from the Killing form is non-degenerate, so that $\mathfrak{u}_{-}$ and $\mathfrak{u}_{+}$ are dual as $\mathfrak{l}$-modules.

\subsection{Quantized enveloping algebras}

We use \cite{klsc} as general reference for this part, even though we keep the conventions of \cite{qflag2}, which use the opposite coproduct.
The \emph{quantized universal enveloping algebra} $U_q(\mathfrak{g})$, for $q \in \mathbb{C}$, is generated by the elements $E_i$, $F_i$, $K_i$, $K_i^{-1}$ with $1\le i\le r$ and $r$ the rank of $\mathfrak{g}$, satisfying certain relations given below. We will only consider the simply-laced case, where the relations take a simpler form. These are
\begin{gather*}
K_i K_i^{-1} = K_i^{-1} K_i=1,\ \
K_i K_j = K_j K_i, \\
K_i E_j K_i^{-1} = q_i^{a_{ij}} E_j,\ \
K_i F_j K_i^{-1} = q_i^{-a_{ij}} F_j, \\
E_i F_j - F_j E_i = \delta_{ij} \frac{K_i - K_i^{-1}}{q_i - q_i^{-1}},
\end{gather*}
plus the quantum analogue of the Serre relations
\begin{gather*}
E_{i} E_{j} - E_{j} E_{i} = 0, \ a_{ij} = 0, \quad
E_{i}^{2} E_{j} - (q + q^{-1}) E_{i} E_{j} E_{i} + E_{j} E_{i}^{2} = 0, \ a_{ij} = - 1, \\
F_{i} F_{j} - F_{j} F_{i} = 0, \ a_{ij} = 0, \quad
F_{i}^{2} F_{j} - (q + q^{-1}) F_{i} F_{j} F_{i} + F_{j} F_{i}^{2} = 0, \ a_{ij} = - 1.
\end{gather*}
It admits a Hopf algebra structure given by
\begin{gather*}
\Delta(K_i)=K_i\otimes K_i, \quad
\Delta(E_i)=E_i\otimes1+ K_i\otimes E_i, \quad
\Delta(F_i)=F_i\otimes K_i^{-1}+1\otimes F_i, \\
S(K_{i}) = K_{i}^{-1}, \quad
S(E_{i}) = - K_{i}^{-1} E_{i}, \quad
S(F_{i}) = - F_{i} K_{i}, \quad
\varepsilon(K_i)=1, \quad
\varepsilon(E_i)=\varepsilon(F_i)=0.
\end{gather*}
For $q \in \mathbb{R}$ the \emph{compact real form} is defined by
\[
K_{i}^{*} = K_{i} \quad
E_{i}^{*} = K_{i} F_{i}, \quad
F_{i}^{*} = E_{i} K_{i}^{-1}.
\]
We will also use the $q$-numbers, which are defined by $[x] = (q^x - q^{-x}) / (q - q^{-1})$ for $x \in \mathbb{R}$, and their factorials, which are defined by $[n]! = [n] [n - 1] \cdots [1]$ for $n \in \mathbb{N}$.

The finite-dimensional irreducible (Type 1) representations of $U_{q}(\mathfrak{g})$ are labelled by their highest weights $\Lambda$ as in the classical case. We denote these modules by $V(\Lambda)$. For $q \in \mathbb{R}$ there is a Hermitian inner product $(\cdot, \cdot)$ on $V(\Lambda)$, unique up to a positive scalar factor, which is compatible with the compact real form in the sense that
\[
(a v, w) = (v, a^{*} w), \quad
v, w \in V(\Lambda), \ a \in U_{q}(\mathfrak{g}).
\]
We also need the braiding on the category of \emph{Type 1 representations}.
Given two $U_{q}(\mathfrak{g})$-modules $V$ and $W$ of this type, the \emph{braiding} is defined by $\widehat{R}_{V, W} = \tau \circ R_{V, W}$.
Here $\tau$ denotes the flip and $R_{V, W}$ is the specialization of the universal $R$-matrix to the modules $V$ and $W$.

Finally we will need the definition of the \emph{quantized Levi factor} $U_q(\mathfrak{l})$.
Recall that corresponding to the choice of a subset of simple roots $S \subset \Pi$ we have the Levi factor $\mathfrak{l}$, a subalgebra of $\mathfrak{g}$. Then the quantized enveloping algebra of the Levi factor is defined as
\[
U_q(\mathfrak{l}) = \{ \textrm{subalgebra of $U_q(\mathfrak{g})$ generated by $K_i^{\pm 1}$ and $E_j, F_j$ with $j \in S$} \}.
\]
This definition of the quantized Levi factor is given for example in \cite[Section 4]{quantum-flag}.

\section{The classical picture}
\label{sec:classical}

In this section we briefly review the classical picture of the Dolbeault–Dirac operator.
We will also outline the computation of its square, obtaining in this way a rough version of the general Parthasarathy formula \cite{partha}.
The main point here is to stress the crucial steps in the derivation.
This will provide some important guidance for the quantum case.

\subsection{Dolbeault–Dirac operator}
A pair $(\mathfrak{g}, \mathfrak{p})$, where $\mathfrak{g}$ is a complex semisimple Lie algebra and $\mathfrak{p}$ is a parabolic Lie subalgebra, provides an infinitesimal description of the complex manifold $G/P$.
Here $G$ is the (connected, simply-connected) Lie group with Lie algebra $\mathfrak{g}$ and $P$ is the subgroup corresponding to $\mathfrak{p}$.
These spaces are referred to as \emph{generalized flag manifolds} and include the family of projective spaces.
Being complex manifolds, we have the corresponding Dolbeault operators acting on complex differential forms.

Here we will consider the classical analogue of the operator $\eth$ introduced in \cite{qflag2}.
Note that its definition makes use of the parabolic subalgebra $\mathfrak{q} = \mathfrak{l} \oplus \mathfrak{u}_{-}$, instead of the standard one $\mathfrak{p}$.
The corresponding manifolds $G/Q$ and $G/P$ are diffeomorphic as real manifolds, but the induced complex structures are inverse to each other.
The upshot is that in this case the holomorphic tangent space can be identified with $\mathfrak{u}_{+}$ as an $\mathfrak{l}$-module.
For more details on the classical geometric picture we refer to \cite[Section 7]{qflag2}.

Classically $\eth$ coincides with the adjoint of the Dolbeault operator $\bar{\partial}: \Omega^{(0, k)} \to \Omega^{(0, k + 1)}$ acting on anti-holomorphic forms. The adjoint is taken with respect to an invariant Hermitian inner product.
After making the relevant identifications, we can consider $\eth$ as an element of $U(\mathfrak{g}) \otimes \mathrm{Cl}$.
Here $U(\mathfrak{g})$ is the enveloping algebra of $\mathfrak{g}$ and $\mathrm{Cl}$ is the (complex) Clifford algebra of $\mathfrak{u}_{+} \oplus \mathfrak{u}_{-}$, which naturally acts on the exterior algebra $\Lambda(\mathfrak{u}_{+})$.
Note that we consider $\mathfrak{u}_-$ as dual to $\mathfrak{u}_+$ and hence the symmetric bilinear form is the dual pairing.
Explicitely
\[
\eth = \sum_{\xi_{i} \in \Delta(\mathfrak{u}_{+})} E_{\xi_{i}} \otimes \intm{i}.
\]
Here $\{E_{\xi_{i}}\}$ are the root vectors corresponding to the radical roots $\{\xi_{i}\}$, while $\intm{i}$ denotes interior multiplication with respect to the basis of $\mathfrak{u}_{-}$ dual to the root vectors.

The \emph{Dolbeault–Dirac operator} is then defined as
\[
D = \eth + \eth^{*} \in U(\mathfrak{g}) \otimes \mathrm{Cl}.
\]
Here $*$ denotes the adjoint, which algebraically is implemented as a $*$-structure on $U(\mathfrak{g}) \otimes \mathrm{Cl}$. More explicitely, on the $U(\mathfrak{g})$ factor it corresponds to the $*$-structure induced by the compact real form of $\mathfrak{g}$, while on $\mathrm{Cl}$ it is induced by the choice of a Hermitian inner product on $\Lambda(\mathfrak{u}_{+})$.
It can be seen that $D$ acts, up to a scalar, as the Dolbeault–Dirac operator on $G/Q$ formed with respect to the canonical $\mathrm{spin}^{c}$ structure, see \cite[Section 3.4]{dirac-book}.

\subsection{Computation of $D^{2}$}

We will now discuss a simple way to compute the square of $D$ on an irreducible generalized flag manifold $G / P$ (strictly speaking $G / Q$, in our conventions).
To make our life easier, we choose the root vectors in such a way that $B(E_{\xi_{i}}, F_{\xi_{i}}) = 1$ with respect to the Killing form.
Moreover we can choose an orthonormal basis of $\mathfrak{u}_{+}$ in such a way that the adjoint of $\intm{i}$, with respect to the Hermitian inner product on $\Lambda(\mathfrak{u}_{+})$, is given by exterior multiplication, which we denote by $\extm{i}$.
In this way we have
$$
D = \sum_{i} E_{\xi_{i}} \otimes \intm{i} + \sum_{i} F_{\xi_{i}} \otimes \extm{i}.
$$
We have $\eth^{2} = 0$, as a consequence of $\bar{\partial}^{2} = 0$. Then we obtain
$$
D^{2} = \sum_{i, j} E_{\xi_{i}} F_{\xi_{j}} \otimes \intm{i} \extm{j} + \sum_{i, j} F_{\xi_{i}} E_{\xi_{j}} \otimes \extm{i} \intm{j}.
$$
For any elements $A, B \in \mathfrak{g}$, we will use the notation $A \sim B$ if $A = B + C$ with $C \in \mathfrak{l}$.
Recall that we have the commutation relation $[\mathfrak{u}_{+}, \mathfrak{u}_{-}] \subset \mathfrak{l}$. Therefore $[E_{\xi_{i}}, F_{\xi_{j}}] \in \mathfrak{l}$ and so we can write $F_{\xi_{j}} E_{\xi_{i}} \sim E_{\xi_{i}} F_{\xi_{j}}$.
Therefore after relabeling we get
$$
D^{2} \sim \sum_{i, j} E_{\xi_{i}} F_{\xi_{j}} \otimes (\intm{i} \extm{j} + \extm{j} \intm{i}).
$$
Write $D^{2} = D^{2}_{D} + D^{2}_{O}$ depending on whether $i = j$ or $i \neq j$.
We have $\intm{i} \extm{j} + \extm{j} \intm{i} = \delta_{i,j} \mathrm{id}$, since we are using dual bases.
Using these relations we immediately conclude that
$$
D^{2}_{D} \sim \sum_{i} E_{\xi_{i}} F_{\xi_{i}} \otimes 1, \quad
D^{2}_{O} \sim 0.
$$
Since we chose $B(E_{\xi_{i}}, F_{\xi_{i}}) = 1$, we have that $C \sim \sum_{i} E_{\xi_{i}} F_{\xi_{i}}$, where $C$ is the quadratic Casimir of $\mathfrak{g}$.
Therefore we conclude that $D^{2} \sim C \otimes 1$.
Note that this result holds algebraically for $D = \eth + \eth^{*} \in U(\mathfrak{g}) \otimes \mathrm{Cl}$, without making reference to its action on the spinor bundle.

\subsection{Projective spaces}

Projective spaces provide a class of examples of generalized flag manifolds.
Since they will be our main concern in this paper, we fix here our conventions for them.
In the case of the $N$-dimensional projective space our setting is the following.
We consider the simple Lie algebra $\mathfrak{g} = \mathfrak{sl}_{N + 1}$ and remove the simple root $\alpha_{N}$, so that $S = \Pi \backslash \{\alpha_{N}\}$.
Then the Levi factor is $\mathfrak{l} = \mathfrak{gl}_{N}$, with the semisimple part being $\mathfrak{k} = \mathfrak{sl}_{N}$.
The nilradical $\mathfrak{u}_{+}$ can be identified with the fundamental representation of $\mathfrak{sl}_{N}$ as an $\mathfrak{l}$-module.

\begin{remark}
Another possible choice is to remove the simple root $\alpha_{1}$. These two choices are related by a Dynkin diagram automorphism, so that they basically amount to the same thing.
There is a corresponding Hopf algebra automorphim also in the quantum setting.
Therefore in the following we will only consider  the case where we remove the simple root $\alpha_{N}$.
\end{remark}

\section{Symmetric and exterior algebras}\label{sec:algebras}

We start by defining the \emph{quantum symmetric algebra} of $\mathfrak{u}_{+}$, following \cite{bezw}, as
$$
S_{q}(\mathfrak{u}_{+}) = T(\mathfrak{u}_{+}) / \langle \ker (\sigma_{\mathfrak{u}_{+}, \mathfrak{u}_{+}} + \mathrm{id}) \rangle.
$$
Here $T(\mathfrak{u}_{+})$ is the tensor algebra and $\sigma_{\mathfrak{u}_{+}, \mathfrak{u}_{+}}$ is the commutor coming from the coboundary stucture, see \cite{qflag2}. We do not need the precise definition of this structure, since concretely $\ker (\sigma_{\mathfrak{u}_{+}, \mathfrak{u}_{+}} + \mathrm{id})$ is the span of the eigenspaces of the braiding $\widehat{R}_{\mathfrak{u}_{+}, \mathfrak{u}_{+}}$ with \emph{negative} eigenvalues.
Similarly we define the the \emph{quantum exterior algebra} of $\mathfrak{u}_{+}$ as
$$
\Lambda_{q}(\mathfrak{u}_{+}) = T(\mathfrak{u}_{+}) / \langle \ker (\sigma_{\mathfrak{u}_{+}, \mathfrak{u}_{+}} - \mathrm{id}) \rangle.
$$
Here $\ker (\sigma_{\mathfrak{u}_{+}, \mathfrak{u}_{+}} - \mathrm{id})$ coincides with the span of the eigenspaces of the braiding $\widehat{R}_{\mathfrak{u}_{+}, \mathfrak{u}_{+}}$ with \emph{positive} eigenvalues.
In this section we will find explicit relations for these algebras.

\subsection{The R-matrix}

We will consider the case of the fundamental representation of $U_{q}(\mathfrak{sl}_{N})$.
The corresponding $R$-matrix (up to rescaling) is given in \cite[Section 8.4.2]{klsc} as
\[
R = q \sum_{i = 1}^{N} E_{ii} \otimes E_{ii} + \sum_{i \neq j} E_{ii} \otimes E_{jj} + (q - q^{-1}) \sum_{i > j} E_{ij} \otimes E_{ji}.
\]
Here $E_{ij}$ is the matrix which is equal to $1$ in the $(i,j)$-position and $0$ otherwise. This is with respect to a weight basis $\{ e_{i} \}_{i = 1}^{N}$, where $e_{1}$ is the highest weight vector.
The braiding is given by $\widehat{R} = \tau \circ R$, where $\tau$ is the flip. It follows
that for $\widehat{R}$ we have
\begin{gather*}
\widehat{R} (e_{i} \otimes e_{i}) = q e_{i} \otimes e_{i}, \quad
\widehat{R} (e_{i} \otimes e_{j}) = e_{j} \otimes e_{i}, \quad i > j,\\
\widehat{R} (e_{i} \otimes e_{j}) = e_{j} \otimes e_{i} + (q - q^{-1}) e_{i} \otimes e_{j}, \quad i < j.
\end{gather*}
Using these formulae we can easily compute its  eigenvectors and eigenvalues.
\begin{lemma}
\label{lem:rmatrix-eigen}
The operator $\widehat{R}$ has the eigenvectors $\{e_{i} \otimes e_{i}\}_{i = 1}^{N}$ and $\{ e_{i} \otimes e_{j} + q^{-1} e_{j} \otimes e_{i} \}_{i < j}$ with eigenvalues $q$ and $\{e_{i}\otimes e_{j} - q e_{j} \otimes e_{i}\}_{i < j}$ with eigenvalues $-q^{-1}$.
\end{lemma}

\begin{proof}
Clearly $e_{i} \otimes e_{i}$ is an eigenvector with eigenvalue $q$. For $i < j$ we compute
\[
\begin{split}
\widehat{R}(e_{i} \otimes e_{j} + q^{-1} e_{j} \otimes e_{i})
& = e_{j} \otimes e_{i} + (q - q^{-1}) e_{i} \otimes e_{j} + q^{-1} e_{i} \otimes e_{j}\\
& = q(e_{i} \otimes e_{j} + q^{-1} e_{j} \otimes e_{i}).
\end{split}
\]
Therefore $e_{i} \otimes e_{j} + q^{-1} e_{j} \otimes e_{i}$ has eigenvalue $q$. Similarly for $i < j$ we compute
\[
\begin{split}
\widehat{R}(e_{i} \otimes e_{j} - q e_{j} \otimes e_{i})
& = e_{j} \otimes e_{i} + (q - q^{-1}) e_{i} \otimes e_{j} - q e_{i} \otimes e_{j}\\
& = -q^{-1} (e_{i} \otimes e_{j} - q e_{j} \otimes e_{i}).
\end{split}
\]
Therefore $e_{i} \otimes e_{j} - q e_{j} \otimes e_{i}$ has eigenvalue $-q^{-1}$.
\end{proof}

\begin{remark}
\label{rmk:rescale-basis}
Rescaling the weight basis $\{e_i\}_{i = 1}^N$ does not change the form of the eigenvectors and the eigenvalues. As a consequence, the relations for the symmetric and exterior algebras will not depend on this rescaling. This can be seen explicitely below.
\end{remark}

\subsection{The algebras}

We identify $\mathfrak{u}_{+}$ with the fundamental representation of $U_{q}(\mathfrak{sl}_{N})$.
We write $\Lambda_{q}^{2} \mathfrak{u}_{+} = \ker (\sigma_{\mathfrak{u}_{+}, \mathfrak{u}_{+}} + \mathrm{id})$ and $S_{q}^{2} \mathfrak{u}_{+} = \ker (\sigma_{\mathfrak{u}_{+}, \mathfrak{u}_{+}} - \mathrm{id})$, since they can be identified with the antisymmetric and symmetric $2$-tensors, respectively.
Using \cref{lem:rmatrix-eigen} we have
\begin{gather*}
\Lambda_{q}^{2} \mathfrak{u}_{+} = \mathrm{span} \{e_{i} \otimes e_{j} - q e_{j} \otimes e_{i}: 1 \leq i < j \leq N\}, \\
S_{q}^{2}\mathfrak{u}_{+} = \mathrm{span} \{e_{i} \otimes e_{i}\}_{i = 1}^{N} \cup \mathrm{span} \{e_{i} \otimes e_{j} + q^{-1} e_{j} \otimes e_{i}: 1 \leq i < j \leq N\}.
\end{gather*}
From these we obtain the relations for the symmetric and exterior algebras. We denote the product in the former by juxtaposition and in the latter by $\wedge$, as usual.

\begin{proposition}
\label{prop:exterior-rels}
The generators of the algebras $S_{q}(\mathfrak{u}_{+})$ and
$\Lambda_{q}(\mathfrak{u}_{+})$ satisfy the relations
\begin{gather*}
e_{i} e_{j} = q e_{j} e_{i}, \quad i < j,\\
e_{i} \wedge e_{i} = 0, \quad e_{i} \wedge e_{j}= - q^{-1} e_{j} \wedge e_{i}, \quad i < j.
\end{gather*}
\end{proposition}

\begin{proof}
The two algebras are defined by $S_{q}(\mathfrak{u}_{+}) = T(\mathfrak{u}_{+}) / \langle \Lambda_{q}^{2} \mathfrak{u}_{+} \rangle$ and $\Lambda_{q} (\mathfrak{u}_{+}) = T(\mathfrak{u}_{+}) / \langle S_{q}^{2} \mathfrak{u}_{+} \rangle$.
Then the relations follow from the explicit characterization of the two subspaces.
\end{proof}

Recall that, given a quadratic algebra $A = T(V) / \langle R \rangle$ with relations $R \subset V \otimes V$, the \emph{quadratic dual algebra} $A^{!}$ is defined by $A^{!} = T(V^{*}) / \langle R^{\circ} \rangle$, where $R^{\circ} \subset V^{*} \otimes V^{*}$ is defined by
$$
R^{\circ} = \{\varphi \in V^{*} \otimes V^{*}: \langle \varphi, R \rangle = 0\},
$$
where $\langle \cdot, \cdot \rangle$ is the pairing between $V$ and $V^{*}$.

The modules $\mathfrak{u}_{+}$ and $\mathfrak{u}_{-}$ are dual.
We denote the dual basis of $\mathfrak{u}_{-}$ by $\{f_{i}\}_{i=1}^{N}$.
It follows from general arguments that $\Lambda_{q}(\mathfrak{u}_{-})$ is the quadratic dual of $S_{q}(\mathfrak{u}_{+})$, see \cite[Proposition 2.11]{bezw}.
Then using the previous result we can obtain the relations for $\Lambda_{q}(\mathfrak{u}_{-})$.

\begin{proposition}
\label{prop:exterior-dual-rels}
The exterior algebra $\Lambda_{q}(\mathfrak{u}_{-})$ satisfies the relations
$$
f_{i} \wedge f_{i} = 0, \quad f_{i} \wedge f_{j}= - q f_{j} \wedge f_{i} \quad i < j.
$$
\end{proposition}

\begin{proof}
The space of relations of $S_{q}(\mathfrak{u}_{+})$ is $\Lambda_{q}^{2}\mathfrak{u}_{+} = \{ e_{i} \otimes e_{j} - q e_{j} \otimes e_{i} \}_{i < j}$. Recall the definition of the pairing $\langle  x \otimes x^{\prime}, y \otimes y^{\prime} \rangle = \langle x^{\prime}, y \rangle \langle x, y^{\prime} \rangle$.
Then for $i < j$ we have
\[
\langle f_{a} \otimes f_{b}, e_{i} \otimes e_{j} - q e_{j} \otimes e_{i} \rangle
 = \langle f_{b}, e_{i} \rangle \langle f_{a}, e_{j} \rangle - q \langle f_{b}, e_{j} \rangle\langle f_{a}, e_{i} \rangle
 = \delta_{ib} \delta_{ja} - q\delta_{jb} \delta_{ia}.
\]
First of all observe that if $a = b$ then this is zero, since $i < j$.
Therefore $f_{i} \otimes f_{i} \in (\Lambda_{q}^{2}\mathfrak{u}_{+})^{\circ}$, which gives the relation $f_{i} \wedge f_{i} = 0$. Then we have
$$
\langle f_{i} \otimes f_{j}, e_{i} \otimes e_{j} - q e_{j} \otimes e_{i} \rangle = -q, \quad\langle f_{j} \otimes f_{i}, e_{i} \otimes e_{j} - q e_{j} \otimes e_{i} \rangle = 1.
$$
Therefore we have $f_{i} \otimes f_{j} + q f_{j}\otimes f_{i} \in (\Lambda_{q}^{2} \mathfrak{u}_{+})^{\circ}$,
since
$$
\langle f_{i} \otimes f_{j} + q f_{j} \otimes f_{i}, e_{i} \otimes e_{j} - q e_{j} \otimes e_{i} \rangle = - q + q = 0.
$$
This gives the relation $f_{i} \wedge f_{j} = -q f_{j} \wedge f_{i}$ for $i < j$.
\end{proof}

\section{Pairings and inner products}
\label{sec:pairing}

In this section we will discuss how to extend the pairing of $\mathfrak{u}_{-}$ and $\mathfrak{u}_{+}$ to a pairing of the corresponding exterior algebras.
Similarly we will extend the unique (up to a constant) invariant Hermitian inner product on $\mathfrak{u}_{+}$ to the whole exterior algebra.
We will follow the approach of \cite{qflag2}, with certain minor modifications.
In their treatment, a normalization in each degree is fixed from the outset, see for example \cite[Proposition 3.6]{qflag2}.
This leads to some small issues in the classical limit.
Instead we will consider general rescalings in each degree, parametrized by certain numbers $\{\lambda_k\}_{k = 0}^N$ and $\{\lambda^\prime_k\}_{k = 0}^N$.
This general setting will turn out to be crucial for the computation of the square of the Dolbeault–Dirac operator.

\subsection{Extension of the pairing}

In the classical case, given a dual pairing $\langle \cdot, \cdot \rangle: V^{*} \otimes V \to\mathbb{C}$ of vector spaces, its standard extension to the corresponding exterior algebras $\langle \cdot, \cdot \rangle: \Lambda^{k}(V^{*}) \otimes \Lambda^{k}(V) \to \mathbb{C}$ is obtained via the determinant. Concretely we have
\[
\langle y_{k} \wedge \cdots \wedge y_{1}, x_{1} \wedge \cdots \wedge x_{k} \rangle
= \det(\langle y_{i}, x_{j} \rangle).
\]
Taking dual bases $\{e_{i}\}_i$ and $\{f_{i}\}_i$ of $V$ and $V^{*}$, that is $\langle f_{i}, e_{j}\rangle=\delta_{ij}$, we see that
\[
\langle f_{k} \wedge \cdots \wedge f_{1}, e_{1}\wedge \cdots \wedge e_{k}\rangle = 1.
\]
This pairing has the property that $\langle y \wedge z, x \rangle = \langle y, \delta_{z}(x) \rangle$, where $\delta_{z}$ denotes interior multiplication by $z$. It is given explicitely by te formula
$$
\delta_{z}(x_{1} \wedge \cdots \wedge x_{k}) = \sum_{j = 1}^{k} (-1)^{j - 1} \langle z, x_{j} \rangle x_{1} \wedge \cdots \wedge \widehat{x_{j}} \wedge \cdots \wedge x_{k},
$$
where the hat denotes omission as usual.
This property is important in order to construct a representation of the Clifford algebra $\mathrm{Cl}(V \oplus V^{*})$ on $\Lambda^{k}(V)$.

A pairing of exterior algebras can also be obtained by restriction of a pairing of tensor algebras, since exterior algebras are quotients of tensor algebras.
This is the strategy adopted in \cite[Section 3]{qflag2}. Let $\langle \cdot, \cdot \rangle: \mathfrak{u}_{-} \otimes \mathfrak{u}_{+} \to \mathbb{C}$ denote the dual pairing. It can be extended to a pairing of tensor algebras, denoted by $\langle \cdot, \cdot \rangle_{T}:\mathfrak{u}_{-}^{\otimes k} \otimes \mathfrak{u}_{+}^{\otimes k} \to \mathbb{C}$, by defining
\[
\langle y_{k} \otimes \cdots \otimes y_{1}, x_{1} \otimes \cdots \otimes x_{k} \rangle_{T}
= \langle y_{1}, x_{1} \rangle \cdots \langle y_{k}, x_{k} \rangle.
\]
It is clearly non-degenerate.
Here the subscript $T$ is to stress that it is a pairing of tensor algebras. Notice also the reverse order of the elements in the first slot.

\begin{notation}
Keeping with the notation of \cite{qflag2}, we denote by $\Lambda_{q}^{k} \mathfrak{u}_{\pm}$ the antisymmetric $k$-tensors and by $\Lambda_{q}^{k}(\mathfrak{u}_{\pm})$ the elements of degree $k$ inside the exterior algebra $\Lambda_{q}(\mathfrak{u}_{\pm})$.
\end{notation}

\begin{remark}
It is crucial to keep in mind the distinction between $\Lambda_{q}^{k} \mathfrak{u}_{\pm}$ and $\Lambda_{q}^{k}(\mathfrak{u}_{\pm})$, despite the similar notations. The former are \emph{submodules} of the tensor algebra $T(\mathfrak{u}_{\pm})$, while the latter are \emph{quotients}.
They can be identified, as we will see below.
\end{remark}

For each degree $k$ we have the maps $\pi_{\pm}^{k}: \Lambda_{q}^{k} \mathfrak{u}_{\pm} \to \Lambda_{q}^{k}(\mathfrak{u}_{\pm})$ obtained by composition of the inclusions $\Lambda_{q}^{k}\mathfrak{u}_{\pm} \to \mathfrak{u}_{\pm}^{\otimes k}$ with the quotients $\mathfrak{u}_{\pm}^{\otimes k} \to \Lambda_{q}^{k}(\mathfrak{u}_{\pm})$. They are shown to be isomorphisms in \cite[Proposition 3.2]{ChTS14}.
Then we define a pairing $\langle \cdot, \cdot \rangle_{\Lambda}: \Lambda_{q}^{k}(\mathfrak{u}_{-}) \otimes \Lambda_{q}^{k}(\mathfrak{u}_{+}) \to \mathbb{C}$ by
\[
\langle y, x \rangle_{\Lambda}
= \langle(\pi_{-}^{k})^{-1}(y), (\pi_{+}^{k})^{-1}(x) \rangle_{T}.
\]
Here the subscript $\Lambda$ is to stress that it is a pairing of exterior algebras. It is non-degenerate, as shown in \cite[Proposition 3.6]{qflag2}.
Notice that we are free to make a rescaling in each degree by a constant, without changing the main properties of the pairing.

We point out that in the classical case the pairing $\langle \cdot, \cdot \rangle_{\Lambda}$, as defined above, does not coincide with the pairing $\langle \cdot, \cdot \rangle$, previously introduced via the determinant.
To see this in a simple example, consider $e_{i} \otimes e_{j} - e_{j} \otimes e_{i} \in \Lambda^{2}\mathfrak{u}_{+}$ for $i < j$.
Applying $\pi_{+}^{2}$ we get
\[
\pi_{+}^{2}(e_{i} \otimes e_{j} - e_{j} \otimes e_{i})
= e_{i} \wedge e_{j} - e_{j} \wedge e_{i}
= 2 e_{i} \wedge e_{j}.
\]
Thus we find that $(\pi_{+}^{2})^{-1} (e_{i} \wedge e_{j}) = \frac{1}{2} (e_{i} \otimes e_{j} - e_{j} \otimes e_{i})$.
Repeating the same argument for $\mathfrak{u}_{-}$ we get $(\pi_{-}^{2})^{-1} (f_{j} \wedge f_{i}) = \frac{1}{2} (f_{j} \otimes f_{i} - f_{i}\otimes f_{j})$, where again $\{f_{i}\}$ is the basis dual to $\{e_{i}\}$.
Now we can easily compute the pairing $\langle f_{j} \wedge f_{i}, e_{i} \wedge e_{j}\rangle_{\Lambda}$. It is given by
$$
\langle f_{j} \wedge f_{i}, e_{i} \wedge e_{j}\rangle_{\Lambda}
= \frac{1}{4} \langle f_{j} \otimes f_{i} - f_{i} \otimes f_{j}, e_{i} \otimes e_{j} - e_{j}\otimes e_{i} \rangle_{T}
= \frac{1}{2}.
$$
But this does not coincide with $\langle f_{j} \wedge f_{i}, e_{i} \wedge e_{j} \rangle = 1$.
The situation is similar for higher degrees, with a mismatch given by $k!$ in degree $k$.
A simple fix for this issue is to rescale $\langle(\pi_{-}^{k})^{-1}(y), (\pi_{+}^{k})^{-1}(x) \rangle_{T}$ in each degree by the appropriate normalization.

\subsection{Computation of the pairing}

In this subsection we compute the pairing of the exterior algebras $\Lambda_{q}^{k}(\mathfrak{u}_{+})$ and $\Lambda_{q}^{k}(\mathfrak{u}_{-})$ for all degrees.
First we introduce some handy notation.

\begin{notation}
Write $\underline{i} = (i_{1}, \cdots, i_{k})$, with $i_{1} < \cdots < i_{k}$. Then we define the elements $e_{\underline{i}} = e_{i_{1}} \wedge \cdots \wedge e_{i_{k}} \in \Lambda_{q}^{k}(\mathfrak{u}_{+})$ and $f_{\underline{i}} = f_{i_{k}} \wedge \cdots \wedge f_{i_{1}} \in \Lambda_{q}^{k}(\mathfrak{u}_{-})$.
\end{notation}

It is easy to see, using the relations for the exterior algebras obtained before, that $\{e_{\underline{i}}\}_{\underline{i}}$ and $\{f_{\underline{i}}\}_{\underline{i}}$ form bases for their respective algebras.
Notice also the reverse order in $f_{\underline{i}}$.

\begin{definition}
We define quantum analogs of the antisymmetrization maps by
\[
\mathbf{A}(e_{\underline{i}}) = \sum_{p \in S_{k}} \mathbf{A}_p(e_{\underline{i}}) = \sum_{p \in S_{k}} (-q)^{\|p\|} e_{p(i_{1})} \otimes \cdots \otimes e_{p(i_{k})}.
\]
An analogous definition is given for the antisymmetrization of the elements $\{ f_{\underline{i}} \}_{\underline{i}}$.
\end{definition}

Indeed, as we will show in the next lemma, they are antisymmetric tensors. Moreover they essentially lift the bases $\{e_{\underline{i}}\}$ and $\{f_{\underline{i}}\}$ to their respective tensor algebras.

\begin{lemma}
We have $\mathbf{A}(e_{\underline{i}}) \in \Lambda_{q}^{k} \mathfrak{u}_{+}$ and $\mathbf{A}(f_{\underline{i}}) \in \Lambda_{q}^{k} \mathfrak{u}_{-}$. Moreover
\[
\pi_{+}^{k}(\mathbf{A}(e_{\underline{i}})) = q^{k (k - 1) / 2} [k]! e_{\underline{i}}, \quad
\pi_{-}^{k}(\mathbf{A}(f_{\underline{i}})) = q^{k (k - 1) / 2} [k]! f_{\underline{i}}.
\]
\end{lemma}

\begin{proof}
Recall that the antisymmetric tensors are defined by $\Lambda_{q}^{k}\mathfrak{u}_{+} = \bigcap_{j = 1}^{k - 1}\ker(\sigma_{j} + \mathrm{id})$,
where $\sigma_{j}$ is the commutor acting on the factors $j$ and $j + 1$, see \cite[Definition 3.5]{qflag2}.
Fix $j \leq k - 1$. Then for any $p \in S_{k}$ there is a unique permutation $p^{\prime} \in S_{k}$ such that $p^{\prime}(i_{j}) = p(i_{j + 1})$ and $p^{\prime}(i_{j + 1}) = p(i_{j})$, with all other elements fixed. We have $\|p^{\prime}\| = \|p\| + 1$ if $p(i_{j}) < p(i_{j + 1})$ and $\|p^{\prime}\| = \|p\| - 1$ if $p(i_{j}) > p(i_{j + 1})$. Consider the first case. We compute
\[
\begin{split}
\mathbf{A}_p(e_{\underline{i}}) + \mathbf{A}_{p^\prime}(e_{\underline{i}})
& =(-q)^{\|p\|} e_{p(i_{1})} \otimes \cdots\otimes e_{p(i_{j})} \otimes e_{p(i_{j + 1})}\otimes \cdots \otimes e_{p(i_{k})}\\
& + (-q)^{\|p^{\prime}\|} e_{p^{\prime}(i_{1})} \otimes \cdots \otimes e_{p^{\prime}(i_{j})} \otimes e_{p^{\prime}(i_{j + 1})} \otimes \cdots\otimes e_{p^{\prime}(i_{k})}\\
& =(-q)^{\|p\|} e_{p(i_{1})} \otimes \cdots\otimes(e_{p(i_{j})} \otimes e_{p(i_{j + 1})}- q e_{p(i_{j + 1})} \otimes e_{p(i_{j})}) \otimes\cdots \otimes e_{p(i_{k})}.
\end{split}
\]
Observe that the element $e_{p(i_{j})} \otimes e_{p(i_{j + 1})}-qe_{p(i_{j + 1})} \otimes e_{p(i_{j})}$ belongs to $\Lambda_{q}^{2}\mathfrak{u}_{+}$, since $p(i_{j}) < p(i_{j + 1})$. The other case is similarly verified. Since this holds for any $p \in S_k$ and we are summing over all permutations, we conclude that $\mathbf{A}(e_{\underline{i}}) \in \Lambda_{q}^{k}\mathfrak{u}_{+}$.

Similarly we consider the case $\mathbf{A}(f_{\underline{i}})$. Its summands are of the form
\[
\mathbf{A}_p(f_{\underline{i}}) = (-q)^{\|p\|} f_{p(i_{k})} \otimes \cdots \otimes f_{p(i_{1})}.
\]
We can repeat the argument used for $\mathbf{A}(e_{\underline{i}})$. Consider the unique permutation $p^\prime \in S_k$ which swaps $p(i_{j})$ and $p(i_{j -1})$. Then we have $\|p^{\prime}\| = \|p\| + 1$ if $p(i_{j}) > p(i_{j - 1})$ and $\|p^{\prime}\| = \|p\| - 1$ if $p(i_{j}) < p(i_{j - 1})$.
This is because the order defining $f_{\underline{i}}$ is the opposite of the order defining $e_{\underline{i}}$.
Let us consider the first case. Summing the relevant terms we get
\[
\begin{split}
\mathbf{A}_p(f_{\underline{i}}) + \mathbf{A}_{p^\prime}(f_{\underline{i}})
& =(-q)^{\|p\|}f_{p(i_{k})}\otimes\cdots\otimes f_{p(i_{j})}\otimes f_{p(i_{j-1})}\otimes\cdots\otimes f_{p(i_{1})}\\
 & +(-q)^{\|p^{\prime}\|}f_{p^{\prime}(i_{k})}\otimes\cdots\otimes f_{p^{\prime}(i_{j})}\otimes f_{p^{\prime}(i_{j-1})}\otimes\cdots\otimes f_{p^{\prime}(i_{1})}\\
 & =(-q)^{\|p\|}f_{p(i_{k})}\otimes\cdots\otimes(f_{p(i_{j})}\otimes f_{p(i_{j-1})}-qf_{p(i_{j-1})}\otimes f_{p(i_{j})})\otimes\cdots\otimes f_{p(i_{1})}.
\end{split}
\]
Since $p(i_{j}) > p(i_{j - 1})$, the element $f_{p(i_{j})} \otimes f_{p(i_{j - 1})} -q f_{p(i_{j - 1})} \otimes f_{p(i_{j})}$
belongs to $\Lambda_{q}^{2}\mathfrak{u}_{-}$. The other case is similar. Therefore we conclude that $\mathbf{A}(f_{\underline{i}}) \in \Lambda_{q}^{k} \mathfrak{u}_{-}$.

Next we compute the elements $\pi_{+}^{k}(\mathbf{A}(e_{\underline{i}}))$. By definition we have
\[
\pi_{+}^{k}(\mathbf{A}(e_{\underline{i}})) = \sum_{p \in S_{k}} (-q)^{\|p\|} e_{p(i_{1})} \wedge \cdots\wedge e_{p(i_{k})}.
\]
We have the relation $e_{a} \wedge e_{b} = -q e_{b} \wedge e_{a}$ for $a > b$. Therefore we can rewrite every element $e_{p(i_{1})} \wedge \cdots \wedge e_{p(i_{k})}$ as $e_{i_{1}} \wedge \cdots \wedge e_{i_{k}}$, up to some constant. Recall that the length $\|p\|$ coincides with the number of inversions. Hence upon reordering the element $e_{p(i_{1})} \wedge \cdots \wedge e_{p(i_{k})}$ we pick up a factor $(-q)^{\|p\|}$. We are left with the expression
\[
\pi_{+}^{k}(\mathbf{A}(e_{\underline{i}})) = \sum_{p \in S_{k}} q^{2 \|p\|} e_{i_{1}} \wedge \cdots\wedge e_{i_{k}}.
\]
The above sum is well known and the result is
\[
\pi_{+}^{k}(\mathbf{A}(e_{\underline{i}})) = q^{k (k - 1) / 2} [k]! e_{i_{1}} \wedge \cdots\wedge e_{i_{k}}.
\]
The case of $\pi_{-}^{k}(\mathbf{A}(f_{\underline{i}}))$ is completely analogous and therefore omitted.
\end{proof}

We are now in the position to compute the pairings of $\{e_{\underline{i}}\}_{\underline{i}}$ with $\{f_{\underline{i}}\}_{\underline{i}}$.

\begin{proposition}
Let $f_{\underline{i}} \in \Lambda_{q}^{k}(\mathfrak{u}_{-})$ and $e_{\underline{j}} \in \Lambda_{q}^{k}(\mathfrak{u}_{+})$.
Then we have $\langle f_{\underline{i}}, e_{\underline{j}}\rangle = 0$ for $\underline{i} \neq\underline{j}$, while $\langle f_{\underline{i}}, e_{\underline{i}}\rangle = q^{- k (k - 1) / 2} / [k]!$ for $\underline{i} = \underline{j}$.
\end{proposition}

\begin{proof}
Recall that the pairing is defined by
\[
\langle f_{\underline{i}},e_{\underline{j}}\rangle=\langle(\pi_{-}^{k})^{-1}(f_{\underline{i}}),(\pi_{+}^{k})^{-1}(e_{\underline{j}})\rangle.
\]
From the previous lemma we have
\[
(\pi_{-}^{k})^{-1} (f_{\underline{i}}) = \frac{q^{- k (k - 1) / 2}}{[k]!} \mathbf{A} (f_{\underline{i}}), \quad
(\pi_{+}^{k})^{-1} (e_{\underline{j}}) = \frac{q^{- k (k - 1) / 2}}{[k]!} \mathbf{A} (e_{\underline{j}}).
\]
Plugging in these expressions and using the definition of the antisymmetrizers we find
\[
\langle f_{\underline{i}}, e_{\underline{j}} \rangle = \frac{q^{- k (k - 1)}}{[k]!^2} \sum_{p, p^{\prime} \in S_{k}}(-q)^{\|p\| + \|p^{\prime}\|} \langle f_{p(i_{k})} \otimes \cdots \otimes f_{p(i_{1})}, e_{p^{\prime}(j_{1})}\otimes \cdots \otimes e_{p^{\prime}(j_{k})} \rangle.
\]
Using the definition of the pairing $\langle\cdot, \cdot \rangle: \mathfrak{u}_{-}^{\otimes k} \otimes \mathfrak{u}_{+}^{\otimes k} \to \mathbb{C}$ we get
\[
\langle f_{p(i_{k})}\otimes\cdots\otimes f_{p(i_{1})},e_{p^{\prime}(j_{1})}\otimes\cdots\otimes e_{p^{\prime}(j_{k})}\rangle=\langle f_{p(i_{1})},e_{p^{\prime}(j_{1})}\rangle\cdots\langle f_{p(i_{k})},e_{p^{\prime}(j_{k})}\rangle.
\]
Recall that $\langle f_{i},e_{j}\rangle=\delta_{ij}$, since they are dual bases. Therefore it is clear that the pairing is zero unless $\underline{i} = \underline{j}$.
Let us consider this case. The only non-zero terms in the sums are those for which we have $p = p^{\prime}$, by the same argument. Hence we get
\[
\langle f_{p(i_{k})} \otimes \cdots\otimes f_{p(i_{1})}, e_{p(i_{1})} \otimes \cdots \otimes e_{p(i_{k})}\rangle = 1.
\]
Therefore we obtain the expression
\[
\langle f_{\underline{i}}, e_{\underline{i}}\rangle = \frac{q^{- k (k - 1)}}{[k]!^2} \sum_{p \in S_{k}} q^{2\|p\|} = \frac{q^{- k (k - 1) / 2}}{[k]!},
\]
where we have used again the sum $\sum_{p \in S_{k}} q^{2\|p\|} = q^{ k (k - 1) / 2} [k]!$.
\end{proof}

\subsection{General pairings and inner products}
By the results of the previous subsection, it follows that we can rescale the pairing $\langle \cdot, \cdot \rangle_{\Lambda}$ in each degree in such a way that $\{e_{\underline{i}}\}_{\underline{i}}$ and $\{f_{\underline{i}}\}_{\underline{i}}$ are dual bases, similarly to the classical case.
More generally we can rescale by any non-zero complex number.
We introduce some special notation to deal with this case.

\begin{notation}
\label{not:dual-pair}
We denote by $\langle \cdot, \cdot \rangle_{k} : \Lambda_{q}^{k}(\mathfrak{u}_{-}) \otimes \Lambda_{q}^{k}(\mathfrak{u}_{+}) \to \mathbb{C}$ the pairing such that
$$
\langle f_{\underline{i}}, e_{\underline{i}} \rangle_{k} = 1, \quad
\langle f_{\underline{i}}, e_{\underline{j}} \rangle_{k} = 0, \  \underline{i} \neq \underline{j}.
$$
More generally let $\{\lambda_k\}_{k = 0}^N$ be some non-zero complex numbers. Then we set
$$
\langle f_{\underline{i}}, e_{\underline{i}} \rangle_{\lambda, k} = \lambda_{k}, \quad
\langle f_{\underline{i}}, e_{\underline{j}} \rangle_{\lambda, k} = 0, \  \underline{i} \neq \underline{j}.
$$
\end{notation}

\begin{remark}
The definition of the Dolbeault–Dirac operator in \cite{qflag2} uses a dual pairing between $\mathfrak{u}_{-}$ and $\mathfrak{u}_{+}$, so that in principle we should set $\lambda_{1} = 1$.
However it is immediate to see that for $\lambda_{1} \neq 1$ we only get an overall constant, which we can always consider.
\end{remark}

Later on we will also need to introduce an Hermitian inner product $(\cdot, \cdot): \Lambda_{q}^{k}(\mathfrak{u}_{+}) \otimes \Lambda_{q}^{k}(\mathfrak{u}_{+}) \to \mathbb{C}$. This can be first defined on $\mathfrak{u}_{+}$ and then extended to the exterior algebra, as discussed previously.
For our applications we should require the property
$$
(x, a \triangleright z) = (a^{*} \triangleright x, z), \quad
x, z \in \Lambda_{q}^{k}(\mathfrak{u}_{+}), \ a \in U_{q}(\mathfrak{l}).
$$
Here $\triangleright$ denotes the action of $U_{q}(\mathfrak{l})$ on the module $\mathfrak{u}_{+}$ (and its extension to the exterior algebra), while $*$ is the involution of $U_{q}(\mathfrak{g})$ coming from the compact real form.

\begin{lemma}
We can choose $(\cdot, \cdot): \mathfrak{u}_{+} \otimes \mathfrak{u}_{+} \to \mathbb{C}$ such that $(e_{i}, e_{j}) = \delta_{ij}$.
\end{lemma}

\begin{proof}
Since $K_{k}^{*} = K_{k}$ we must have $(e_{i}, K_{k} \triangleright e_{j}) = (K_{k} \triangleright e_{i}, e_{j})$ for all $i, j, k$.
Using the fact that $\{e_{i}\}$ is a weight basis we have $K_{k} \triangleright e_{i} = q^{(\alpha_{k}, \beta_{i})} e_{i}$, where $\beta_{i}$ denotes the weight of $e_{i}$.
Plugging this into the previous relation we get  $q^{(\alpha_{k}, \beta_{j})} (e_{i}, e_{j}) = q^{(\alpha_{k}, \beta_{i})} (e_{i}, e_{j})$.
Since the weights are distinct, this implies that $(e_{i}, e_{j}) = 0$ for $i \neq j$. Finally we can always rescale the weight basis, as in \cref{rmk:rescale-basis}, to set $(e_{i}, e_{i}) = 1$ for all $i$.
\end{proof}

The extension of this Hermitian inner product to the exterior algebra proceeds exactly as for the dual pairing.
As in that case, we introduce some notation.

\begin{notation}
\label{not:herm-prod}
We denote by $(\cdot, \cdot)_k : \Lambda_{q}^{k}(\mathfrak{u}_{+}) \otimes \Lambda_{q}^{k}(\mathfrak{u}_{+}) \to \mathbb{C}$ the Hermitian product such that
$$
(e_{\underline{i}}, e_{\underline{i}})_{k} = 1, \quad
(e_{\underline{i}}, e_{\underline{j}})_{k} = 0, \  \underline{i} \neq \underline{j}.
$$
More generally let $\{\lambda_k^\prime\}_{k = 0}^N$ be some non-zero positive numbers. Then we set
$$
(e_{\underline{i}}, e_{\underline{i}})_{\lambda^{\prime}, k} = \lambda_{k}^{\prime}, \quad
(e_{\underline{i}}, e_{\underline{j}})_{\lambda^{\prime}, k} = 0, \  \underline{i} \neq \underline{j}.
$$
\end{notation}

\begin{remark}
It can be seen that $\Lambda_{q}^{k}(\mathfrak{u}_{\pm})$ are irreducible $U_{q}(\mathfrak{sl}_{N})$-modules, so that the choice of the pairing $\langle \cdot, \cdot \rangle$ and of the Hermitian inner product $(\cdot, \cdot)$ in each degree is unique up to a constant.
Therefore the numbers $\{\lambda_k\}_{k = 0}^N$ and $\{\lambda_k^\prime\}_{k = 0}^N$ parametrize all such choices.
\end{remark}

In the following we will refer to the special choices $\langle \cdot, \cdot\rangle_{k}$ and $(\cdot, \cdot)_{k}$ as the \emph{normalized} pairings and the inner products respectively.
These will make certain computations easier.
On the other hand, the introduction of the rescalings $\{\lambda_k\}_{k = 0}^N$ and $\{\lambda^\prime_k\}_{k = 0}^N$ will play a crucial role in the computation of the square of the Dolbeault–Dirac operator.

\section{The quantum Clifford algebra}
\label{sec:clifford}

In this section we will consider the quantum Clifford algebra with $\mathfrak{u}_{+}$ being the fundamental representation of $\mathfrak{sl}_{N}$, following the general setting of \cite{qflag2}.
We will obtain commutation relations for the generators.
These relations will depend on the choice of the pairings and the inner products, as introduced in the previous section.
At first we will perform the computations in the normalized case. The general case will then be obtained by rescaling.

\subsection{Definitions}

According to \cite[Definition 5.2]{qflag2} the quantum Clifford algebra can be defined as $\mathrm{End}(\Lambda_{q}(\mathfrak{u}_{+}))$ together with a certain factorization given by two maps $\gamma_{+}$ and $\gamma_{-}$.
The map $\gamma_{+}$ is simply left multiplication of the algebra $\Lambda_{q}(\mathfrak{u}_{+})$ on itself, that is
$$\gamma_{+}(x) z = x \wedge z, \quad x, z \in \Lambda_{q}(\mathfrak{u}_{+}).$$
The other map $\gamma_{-}$ is defined as follows. With $\langle \cdot, \cdot \rangle$ denoting the pairing between $\Lambda_{q}(\mathfrak{u}_{-})$ and $\Lambda_{q}(\mathfrak{u}_{+})$, the action $\gamma_{-}$ of $\Lambda_{q}(\mathfrak{u}_{-})$ on $\Lambda_{q}(\mathfrak{u}_{+})$ is given by
$$
\langle w, \gamma_{-}(y) x \rangle = \langle w \wedge y, x\rangle, \quad y, w \in \Lambda_{q}(\mathfrak{u}_{-}), \quad x \in \Lambda_{q}(\mathfrak{u}_{+}).
$$
It can be proven that the map $\Lambda_{q}(\mathfrak{u}_{-}) \otimes \Lambda_{q}(\mathfrak{u}_{+}) \to \mathrm{End}(\Lambda_{q}(\mathfrak{u}_{+}))$ given by $y \otimes x \to \gamma_{-}(y) \gamma_{+}(x)$ is an isomorphism of $U_{q}(\mathfrak{l})$-modules, see \cite[Theorem 5.1]{qflag2}.

\subsection{The operator $\gamma_{-}$}

For the rest of this section we fix the pairing $\langle \cdot, \cdot \rangle_{k}$ defined in \cref{not:dual-pair}, since the relevant computations are easier in this case. Later on we will obtain the relations for the general case by simple rescalings.
We start by computing the explicit action of $\gamma_{-}(f_{a})$ on the basis elements $\{e_{\underline{i}}\}_{\underline{i}}$ defined previously.

\begin{lemma}
\label{lem:cliff1}
Let $e_{\underline{i}} \in \Lambda_{q}^{k}(\mathfrak{u}_{+})$. Then
we have 
\[
\gamma_{-}(f_{a})e_{\underline{i}}=\sum_{r=1}^{k}\delta_{a,i_{r}}(-q)^{r-1}e_{\underline{i}\backslash i_{r}}.
\]
\end{lemma}

\begin{proof}
Let $f_{\underline{j}} \in \Lambda_{q}^{k - 1} (\mathfrak{u}_{-})$.
By definition of $\gamma_{-}$ we have
\[
\langle f_{\underline{j}},\gamma_{-}(f_{a})e_{\underline{i}}\rangle_{k-1}=\langle f_{\underline{j}}\wedge f_{a},e_{\underline{i}}\rangle_{k}.
\]
First consider the case $a\cap\underline{i} =  \emptyset$. Using the definition of the pairing we find $\langle f_{\underline{j}} \wedge  f_{a}, e_{\underline{i}} \rangle_{k} = 0$, since $e_{\underline{i}}$ does not contain the element $e_{a}$.
On the other hand consider $a \cap \underline{i} \neq \emptyset$, that is $a = i_{r}$ for some $r$. Then the pairing is non-zero if $\underline{j} = \underline{i} \backslash i_{r}$.
Therefore we can write
\[
\langle f_{\underline{j}},\gamma_{-}(f_{a})e_{\underline{i}}\rangle_{k}=\sum_{r=1}^{k}\delta_{a,i_{r}}\delta_{\underline{j},\underline{i}\backslash i_{r}}\langle f_{\underline{i}\backslash i_{r}}\wedge f_{a},e_{\underline{i}}\rangle_{k}.
\]
Recall that the multi-indices $\underline{i}$ are ordered, so that $i_{c} < i_{r}$ for $c < r$.
From \cref{prop:exterior-dual-rels} we have the relations $f_{a} \wedge f_{b} = -q f_{b} \wedge f_{a}$ for $a < b$.
Then we compute
\[
\begin{split}\langle f_{\underline{i}\backslash i_{r}}\wedge f_{a},e_{\underline{i}}\rangle_{k} & =\langle f_{i_{k}}\wedge\cdots\wedge\hat{f}_{i_{r}}\wedge\cdots\wedge f_{i_{1}}\wedge f_{i_{r}},e_{\underline{i}}\rangle_{k}\\
 & =(-q)^{r-1}\langle f_{i_{k}}\wedge\cdots\wedge f_{i_{r}}\wedge\cdots\wedge f_{i_{1}},e_{\underline{i}}\rangle_{k}\\
 & =(-q)^{r-1}\langle f_{\underline{i}},e_{\underline{i}}\rangle_{k}=(-q)^{r-1}.
\end{split}
\]
From this we conclude that
\[
\langle f_{\underline{j}},\gamma_{-}(f_{a})e_{\underline{i}}\rangle_{k}=\sum_{r=1}^{k}\delta_{a,i_{r}}\delta_{\underline{j},\underline{i}\backslash i_{r}}(-q)^{r-1}.
\]
Finally expanding in terms of the the dual bases we find the expression
\[
\gamma_{-}(f_{a}) e_{\underline{i}} = \sum_{j_{1} < \cdots < j_{k}} \langle f_{\underline{j}}, \gamma_{-} (f_{a}) e_{\underline{i}} \rangle_{k} e_{\underline{j}} = \sum_{r = 1}^{k} \delta_{a, i_{r}} (-q)^{r - 1} e_{\underline{i} \backslash i_{r}}.
\qedhere
\]
\end{proof}

\begin{remark}
For $q = 1$ we recover the classical expression
\[
\gamma_{-}(f_{a}) e_{\underline{i}} = \sum_{r = 1}^{k} (-1)^{r - 1} \langle f_{a}, e_{i_{r}}\rangle e_{i_{1}} \wedge \cdots \wedge \hat{e}_{i_{r}} \wedge \cdots \wedge e_{i_{k}},
\]
keeping in mind that for the dual pairing we have $\langle f_{a}, e_{i_{r}} \rangle =\delta_{a, i_{r}}$.
\end{remark}

\subsection{The adjoint of $\gamma_{-}$}

In the following we will need commutation relations between the elements $\gamma_{-}(\cdot)$ and their adjoints.
Similarly to the previous subsection, we fix the Hermitian inner product $(\cdot, \cdot)_{k}$ defined in \cref{not:herm-prod} to simplify the computations.
Recall that with this choice $\{e_{\underline{i}}\}_{\underline{i}}$ is an orthonormal basis.
The adjoint of $\gamma_{-}(f_{a})$ is defined as usual by
\[
(\gamma_{-}(f_{a})^{*} x, z)_{k + 1} = (x, \gamma_{-}(f_{a}) z)_{k}, \quad
x \in \Lambda_{q}^{k}(\mathfrak{u}_{+}), \  z \in \Lambda_{q}^{k + 1}(\mathfrak{u}_{+}).
\]
In the next lemma we show that, for this particular choice, we have $\gamma_{-}(f_{a})^{*} = \gamma_{+}(e_{a})$.
This will no longer be the case when we will move to the more general situation.

\begin{lemma}
\label{lem:cliff2}
We have the equality $\gamma_{-}(f_{a})^{*} = \gamma_{+}(e_{a})$.
\end{lemma}

\begin{proof}
Take $e_{\underline{i}} \in \Lambda_{q}^{k}(\mathfrak{u}_{+})$ and
$e_{\underline{j}} \in \Lambda_{q}^{k+1}(\mathfrak{u}_{+})$. Since
the basis is orthonormal we have
\[
\gamma_{-}(f_{a})^{*} e_{\underline{i}} =\sum_{j_{1} < \cdots < j_{k + 1}} (\gamma_{-}(f_{a})^{*} e_{\underline{i}}, e_{\underline{j}})_{k + 1} e_{\underline{j}}.
\]
Using the explicit action of $\gamma_{-}(f_{a})$ we have
\[
(e_{\underline{i}}, \gamma_{-}(f_{a}) e_{\underline{j}})_{k} = \sum_{r = 1}^{k + 1}\delta_{a, j_{r}} (-q)^{r - 1}(e_{\underline{i}}, e_{\underline{j} \backslash j_{r}})_{k}.
\]
The inner product $(e_{\underline{i}}, e_{\underline{j} \backslash j_{r}})_{k}$ is non-zero for $\underline{i} = \underline{j}\backslash j_{r}$.
Notice that this implies that $\underline{i}$ does not contain $a$.
Then for $\underline{i} \cap a \neq \emptyset$ we get $\gamma_{-}(f_{a})^{*} e_{\underline{i}} = 0$.
Consider then $\underline{i} \cap a = \emptyset$. In this case we can rewrite the condition $\underline{i} = \underline{j} \backslash j_{r}$
as $\underline{j} = \underline{i} \cup j_{r} = \underline{i} \cup a$. Then we obtain the expression
\[
(\gamma_{-}(f_{a})^{*} e_{\underline{i}}, e_{\underline{j}})_{k + 1} = \delta_{\underline{j}, \underline{i} \cup a} (-q)^{r-1}.
\]
Plugging back in we get the result
\[
\gamma_{-}(f_{a})^{*} e_{\underline{i}} =\sum_{j_{1} < \cdots < j_{k + 1}}\delta_{\underline{j}, \underline{i} \cup a}(-q)^{r - 1} e_{\underline{j}} = (-q)^{r - 1}e_{\underline{i} \cup a}.
\]

On the other hand consider $\gamma_{+}(e_{a}) e_{\underline{i}} = e_{a} \wedge e_{\underline{i}}$. For $\underline{i} \cap a \neq \emptyset$ this is clearly zero. For $\underline{i} \cap a = \emptyset$ instead
there exists an index $r$ such that
\[
i_{1}<\cdots<i_{r-1}<a<i_{r}<\cdots<i_{k}.
\]
From \cref{prop:exterior-rels} we have $e_{a} \wedge e_{b} = -q^{-1} e_{b} \wedge e_{a}$ for $a < b$. Then we can rewrite
\[
\begin{split}
\gamma_{+}(e_{a}) e_{\underline{i}}
& = e_{a} \wedge e_{i_{1}} \wedge \cdots \wedge e_{i_{r - 1}} \wedge e_{i_{r}} \wedge \cdots\wedge e_{i_{k}}\\
& = (-q)^{r - 1} e_{i_{1}} \wedge \cdots \wedge e_{i_{r}} \wedge e_{a} \wedge e_{i_{r + 1}} \cdots \wedge e_{i_{k}}
= (-q)^{r - 1} e_{\underline{i} \cup a}.
\end{split}
\]
Comparing the two expressions we see that $\gamma_{-}(f_{a})^{*} = \gamma_{+}(e_{a})$.
\end{proof}

\subsection{Commutation relations}

In this subsection we will obtain commutation relations for the operators $\gamma_{-}(f_{i})$ and $\gamma_{-}(f_{j})^{*}$.
We adopt the following notations.

\begin{notation}
We write $\intm{i} = \gamma_{-}(f_{i})$ and $\extm{j} = \gamma_{-}(f_{j})^{*}$, where the letters are chosen to remind the reader that they correspond to interior and exterior multiplication, respectively.
\end{notation}

We start with the easier case, that is $i \neq j$.

\begin{proposition}
For $i \neq j$ we have $\extm{i} \intm{j} = - q^{-1} \intm{j} \extm{i}$.
\end{proposition}

\begin{proof}
First we act with $\extm{i} \intm{j}$ on a basis vector $e_{\underline{k}}$. Suppose that $(\underline{k} \backslash j) \cap i  = \emptyset$ and $\underline{k} \cap j \neq \emptyset$, since otherwise the result is zero.
Let $r$ be the position of $i$ inside $(\underline{k} \backslash j) \cup i$
and $s$ be the position of $j$ inside $\underline{k}$. Using \cref{lem:cliff1} and \cref{lem:cliff2} we compute
\[
\extm{i} \intm{j} e_{\underline{k}}
= (-q)^{s - 1} \extm{i} e_{\underline{k} \backslash j}
= (-q)^{r - 1} (-q)^{s - 1} e_{(\underline{k} \backslash  j) \cup i}.
\]
Similarly we act with $\intm{j} \extm{i}$ on $e_{\underline{k}}$.
In this case let $r^{\prime}$ be the position of $i$ inside $\underline{k} \cup i$ and $s^{\prime}$ be the position of $j$ inside $\underline{k} \cup i$.
Then we have
\[
\intm{j} \extm{i} e_{\underline{k}}
= (-q)^{r^{\prime} - 1} \intm{j} e_{\underline{k} \cup i}
= (-q)^{r^{\prime} - 1} (-q)^{s^{\prime} - 1}e_{(\underline{k} \cup i) \backslash j}.
\]
Since $i \neq j$ we clearly have $(\underline{k} \backslash j) \cup i = (\underline{k} \cup i)\backslash j$. Therefore
\[
\extm{i} \intm{j} e_{\underline{k}}
= (-q)^{r - r^{\prime}} (-q)^{s - s^{\prime}}\intm{j} \extm{i} e_{\underline{k}}.
\]

Now suppose $i < j$. Then the positions of $i$ inside $(\underline{k} \backslash j) \cup i$
and $\underline{k} \cup i$ coincide, that is $r^{\prime} = r$. On the other hand the position of $j$ inside $\underline{k} \cup i$ is shifted
by one with respect to its position inside $\underline{k}$, that is $s^{\prime} = s + 1$.
In this case we have
\[
\extm{i} \intm{j} e_{\underline{k}}
= -q^{-1} \intm{j} \extm{i} e_{\underline{k}},\quad i < j.
\]
Similarly consider $i > j$. Then the position of $i$ inside $(\underline{k} \backslash j) \cup i$
is smaller than that inside $\underline{k} \cup i$, that is $r^{\prime} = r + 1$.
On the other hand we have $s^{\prime} = s$. Then also in this case we get
\[
\extm{i} \intm{j} e_{\underline{k}}
= -q^{-1} \intm{j} \extm{i} e_{\underline{k}}, \quad i > j. \qedhere
\]
\end{proof}

The most interesting relations are obtained in the case $i = j$. These are significantly more complicated than the case $i \neq j$.
It is fortunate that also in this case we obtain quadratic-constant relations, since there is no a priori reason why this should be the case.

\begin{proposition}
\label{prop:cliff-iej}
For $i \leq N$ we have
$$\extm{i} \intm{i} - q(q - q^{-1}) \sum_{j = 1}^{i - 1} \extm{j} \intm{j} + \intm{i} \extm{i} = \mathrm{id}.$$
\end{proposition}

\begin{proof}
Denote the above expression by $S_{i}$.
We start with $i = 1$, in which case we have $S_{1} = \extm{1} \intm{1} + \intm{1} \extm{1}$.
We act on the basis vectors $e_{\underline{k}}$, where $\underline{k} = (k_{1}, \cdots, k_{n})$. Consider the case $\underline{k} \cap 1 \neq\emptyset$, so that $\extm{1} e_{\underline{k}} = 0$. We must have $k_{1} = 1$, since $\underline{k}$ is ordered. Then we get
$$
S_{1} e_{\underline{k}}
= \extm{1} \intm{1} e_{\underline{k}}
= \extm{1} e_{\underline{k} \backslash 1}
= e_{\underline{k}}.
$$
Similarly consider the case $\underline{k} \cap 1 = \emptyset$, so that $\intm{1} e_{\underline{k}} = 0$. Then we have
$$
S_{1} e_{\underline{k}}
= \intm{1} \extm{1} e_{\underline{k}}
= \intm{1} e_{\underline{k} \cup 1}
= e_{\underline{k}}.
$$

We proceed by induction over $i$. Observe that by splitting the sum inside $S_{i}$ (and by summing and subtracting the term $\intm{i - 1} \extm{i - 1}$) we can write
\[
\begin{split}S_{i}
& = \extm{i - 1} \intm{i - 1} - q (q - q^{-1}) \sum_{j = 1}^{i - 2} \extm{j} \intm{j} + \mathfrak{i}_{i-1}\mathfrak{e}_{i-1}\\
& -q^{2} \extm{i - 1} \intm{i - 1} - \intm{i - 1} \extm{i - 1} + \extm{i} \intm{i} + \intm{i} \extm{i}.
\end{split}
\]
The first line coincides with $S_{i - 1}$, so that we have the identity
$$
S_{i} = S_{i-1} - q^{2} \extm{i - 1} \intm{i - 1} - \intm{i - 1} \extm{i - 1} + \extm{i} \intm{i} + \intm{i} \extm{i}.
$$
By the induction hypothesis we get
\[
S_{i} e_{\underline{k}}
= (1 - q^{2} \extm{i - 1} \intm{i - 1} - \intm{i - 1} \extm{i - 1} + \extm{i} \intm{i} + \intm{i} \extm{i}) e_{\underline{k}}.
\]
As for the case $i = 1$ we distinguish between $\underline{k} \cap i \neq \emptyset$
and $\underline{k} \cap i = \emptyset$.
First consider the case $\underline{k} \cap i\neq \emptyset$ and denote
by $r$ the position of $i$ inside $\underline{k}$. Then using the explicit expressions for $\extm{i}$ and $\intm{i}$ we find $\extm{i} \intm{i} e_{\underline{k}} = (-q)^{r - 1} \extm{i} e_{\underline{k} \backslash i} = q^{2(r - 1)} e_{\underline{k}}$
and $\intm{i} \extm{i} e_{\underline{k}} = 0$. Therefore
$$
S_{i} e_{\underline{k}}
= (1 - q^{2} \extm{i - 1} \intm{i - 1} - \intm{i - 1} \extm{i - 1} + q^{2(r - 1)}) e_{\underline{k}}.
$$
Suppose that $\underline{k} \cap (i - 1) \neq \emptyset$. Then the position of $i - 1$ inside $\underline{k}$ is necessarily $r - 1$, since $\underline{k}$ is ordered. Repeating the argument above we
get $\extm{i - 1} \intm{i - 1} e_{\underline{k}} = q^{2(r - 2)} e_{\underline{k}}$
and $\intm{i - 1} \extm{i - 1} e_{\underline{k}} = 0$. Plugging these expressions in we find
$$
S_{i} e_{\underline{k}}
= (1 - q^{2} q^{2(r - 2)} + q^{2(r - 1)}) e_{\underline{k}}
= e_{\underline{k}}.
$$
Conversely consider the case $\underline{k} \cap (i - 1) = \emptyset$.
Then the position of $i - 1$ inside $\underline{k} \cup(i - 1)$ is necessarily
$r$. Then we get $\extm{i - 1} \intm{i - 1} e_{\underline{k}} = 0$ and $\intm{i - 1} \extm{i - 1} e_{\underline{k}} = (-q)^{r - 1}\intm{i - 1} e_{\underline{k} \cup (i - 1)} = q^{2(r - 1)}e_{\underline{k}}$.
Plugging these expressions in we find
$$
S_{i} e_{\underline{k}}
= (1 - q^{2(r - 1)} + q^{2(r - 1)}) e_{\underline{k}}
= e_{\underline{k}}.
$$

The second case $\underline{j}\cap i=\emptyset$ is similar. This becomes clear if we denote again by $r$ the position of $i$ inside
$\underline{j}\cup i$. With this notation we have again
$$
S_{i} e_{\underline{k}}
= (1 - q^{2} \extm{i - 1} \intm{i - 1} -\intm{i - 1} \extm{i - 1} + q^{2(r - 1)}) e_{\underline{k}}.
$$
If $\underline{k} \cap (i - 1) \neq \emptyset$ then the position of $i - 1$ inside $\underline{k}$ is $r - 1$, since it is the same as the position of $i - 1$ inside $\underline{k}\cup i$. On the other hand if $\underline{k}\cap (i - 1) = \emptyset$ then the position of $i - 1$ inside $\underline{k} \cup (i - 1)$ is $r$.
Therefore by the computations of the previous case we get that $S_{i} e_{\underline{k}}= e_{\underline{k}}$.
\end{proof}

\subsection{General relations}

We will now illustrate how to obtain the commutation relations for the general pairings and inner products introduced in \cref{not:dual-pair} and \cref{not:herm-prod}.
Following the general theory, the operator $\gamma_{i} = \gamma_{-}(f_{i})$ and its adjoint are defined by
$$
\langle w, \gamma_{i} x \rangle_{\lambda, k - 1} = \langle w \wedge f_{i}, x \rangle_{\lambda, k}, \quad
(\gamma_{i}^{*} z, x)_{\lambda^{\prime}, k + 1} = (z, \gamma_{i} x)_{\lambda^{\prime}, k},
$$
The next lemma tells us how these operators are related to $\intm{i}$ and $\extm{i}$.

\begin{lemma}
Acting on elements of degree $k$, we have the identities
$$
\gamma_{i} = \frac{\lambda_{k}}{\lambda_{k - 1}}\mathfrak{i}_{i}, \quad
\gamma_{i}^{*} = \frac{\bar{\lambda}_{k + 1}}{\bar{\lambda}_{k}} \frac{\lambda_{k}^{\prime}}{\lambda_{k + 1}^{\prime}} \mathfrak{e}_{i}.
$$
\end{lemma}

\begin{proof}
Using $\langle \cdot, \cdot \rangle_{\lambda, k} = \lambda_{k} \langle \cdot, \cdot \rangle_{k}$
in the defining relation of $\gamma_{i}$ we obtain
\[
\lambda_{k - 1} \langle w, \gamma_{i} x\rangle_{k - 1} = \lambda_{k} \langle w \wedge f_{i}, x \rangle_{k}.
\]
Then using $\langle w, \mathfrak{i}_{i} x\rangle_{k - 1} = \langle w \wedge f_{i}, x\rangle_{k}$
we get
\[
\lambda_{k - 1} \langle w, \gamma_{i} x\rangle_{k - 1} = \lambda_{k} \langle w,\mathfrak{i}_{i} x \rangle_{k - 1}.
\]
From the non-degeneracy of the pairings we conclude that $\gamma_{i} = \frac{\lambda_{k}}{\lambda_{k - 1}} \mathfrak{i}_{i}$.

Similarly consider $\gamma_{i}^{*}$ and the relation $(\gamma_{i}^{*} z, x)_{\lambda^{\prime}, k + 1} = (z, \gamma_{i} x)_{\lambda^{\prime}, k}$.
In the right-hand side the operator $\gamma_{i}$ acts on an element of degree $k + 1$ by $\frac{\lambda_{k + 1}}{\lambda_{k}} \mathfrak{i}_{i}$.
Rescaling the inner products using $(\cdot, \cdot)_{\lambda^{\prime}, k} = \lambda^{\prime}_{k} (\cdot, \cdot)_{k}$ we obtain the identity
\[
\lambda_{k + 1}^{\prime}(\gamma_{i}^{*}x, z)_{k + 1} = \frac{\lambda_{k + 1}}{\lambda_{k}}\lambda_{k}^{\prime}(x, \mathfrak{i}_{i} z)_{k}.
\]
Now using the relation $(x, \mathfrak{i}_{i}z)_{k} = (\mathfrak{e}_{i} x, z)_{k + 1}$ we obtain
\[
(\gamma_{i}^{*} x, z)_{k + 1} = \frac{\lambda_{k + 1}}{\lambda_{k}} \frac{\lambda_{k}^{\prime}}{\lambda_{k + 1}^{\prime}} (\mathfrak{e}_{i} x, z)_{k + 1}.
\]
Recall that the Hermitian inner product $(\cdot, \cdot)_{k}$ is conjugate-linear in the first variable.
Then using non-degeneracy again we conclude that $\gamma_{i}^{*} = \frac{\bar{\lambda}_{k + 1}}{\bar{\lambda}_{k}} \frac{\lambda_{k}^{\prime}}{\lambda_{k + 1}^{\prime}} \mathfrak{e}_{i}$.
\end{proof}

The following identities can be immediately obtained from the previous lemma.

\begin{corollary}
\label{cor:cliff-rescaling}
Set $c_{k} = |\lambda_{k}|^{2} / \lambda_{k}^{\prime}$. Then acting on elements of degree $k$ we have
\begin{equation}
\gamma_{i}^{*} \gamma_{j} = \frac{c_{k}}{c_{k - 1}} \extm{i} \intm{j}, \quad
\gamma_{i} \gamma_{j}^{*} = \frac{c_{k + 1}}{c_{k}} \intm{i} \extm{j}.
\end{equation}
\end{corollary}
These relations will be needed to compute the square of the Dolbeault–Dirac operator.

\section{Quantum root vectors}
\label{sec:root}

In this section we will show some properties satisfied by the quantum root vectors corresponding to the roots of $\mathfrak{u}_{+}$.
First we will obtain commutation relations among them. We will only consider these modulo terms in $U_{q}(\mathfrak{l})$, similarly to our treatment in the classical case.
Then we will obtain explicit expressions for the adjoint action of $U_{q}(\mathfrak{l})$ on them. This will be used to identify the quantum root vectors with the weight basis of $\mathfrak{u}_{+}$ previously introduced.

\subsection{Definitions}

For background material on quantum root vectors we refer the reader to \cite[Section 6.2]{klsc}.
As the name indicates, they provide a quantum counterpart to the usual notion of root vectors for semisimple Lie algebras. They are defined in terms of the Lusztig automorphisms $T_{i}$, which implement the action of the braid group $B_{\mathfrak{g}}$ on $U_{q}(\mathfrak{g})$.
We will only give the formulae which will be relevant for our computations. These are
\begin{gather*}
T_i (E_i) = - F_i K_i, \quad
T_{i}(E_{j}) = E_{j}, \quad a_{ij} = 0,\\
T_{i}(E_{j}) = - E_{i} E_{j} + q^{-1} E_{j} E_{i}, \quad a_{ij} = - 1.
\end{gather*}
Here $a_{ij}$ denote the entries of the Cartan matrix.

Let $w_{0}$ be the longest word of the Weyl group of $\mathfrak{g}$. Let $w_{0} = s_{j_{1}} \cdots s_{j_{k}}$ be a fixed reduced decomposition. Then all the positive roots can be obtained as
$$
\beta_{i} = s_{j_{1}} \cdots s_{j_{i - 1}} (\alpha_{j_i}), \quad i =  1, \cdots, k.
$$
The quantum root vectors are defined similarly as
$$
E_{\beta_{i}} = T_{j_{1}} \cdots T_{j_{i - 1}} (E_{j_i}), \quad i =  1, \cdots, k.
$$
We warn the reader that the quantum root vectors depend on the choice of the reduced decomposition for $w_{0}$ (in the classical case the only ambiguities are signs).

As explained before, in our case we have $\mathfrak{g} = \mathfrak{sl}_{N + 1}$. Excluding the simple root $\alpha_{N}$ we get the Levi factor $\mathfrak{l} = \mathfrak{gl}_{N}$, having semisimple part $\mathfrak{k} = \mathfrak{sl}_{N}$.
For the longest word of $\mathfrak{g} = \mathfrak{sl}_{N + 1}$ we choose the reduced expression
$$
w_{0} = s_{1} (s_{2} s_{1}) \cdots (s_{N-1} \cdots s_{1}) (s_{N} \cdots s_{1}).
$$
Then the longest word of (the semisimple part of) $\mathfrak{l}$ is given by
\[
w_{0,\mathfrak{l}} = s_{1} (s_{2} s_{1}) \cdots (s_{N - 1} \cdots s_{1}).
\]
With these choices we have the factorization $w_{0} = w_{0, \mathfrak{l}}w_{\mathfrak{l}}$, where $w_{\mathfrak{l}} = s_{N} \cdots s_{1}$, which is in accordance with \cite[Convention 3.13]{qflag2}.

We denote the radical roots $\Delta(\mathfrak{u}_{+})$ by $\{\xi_{i}\}_i$. Write $w_{0,\mathfrak{l}} = s_{j_{1}} \cdots s_{j_{m}}$ and $w_{\mathfrak{l}} = s_{j_{m + 1}} \cdots s_{j_{m + n}}$.
Then the radical roots can be obtained \cite[Lemma 3.15]{qflag2} by
$$
\xi_{i} = s_{j_{1}} \cdots s_{j_{m}} s_{j_{m + 1}} \cdots s_{j_{m + i - 1}} (\alpha_{j_{m + i}}), \quad
i = 1, \cdots, n.
$$

The computation of the radical roots is then a completely classical problem. To obtain simple expression for the quantum root vectors, on the other hand, we need to make use of the braid relations satisfied by the automorphisms $T_{i}$.
We do not give the details concerning this computation but simply report the results, which are
$$
\xi_{i} = \sum_{j = i}^{N} \alpha_{j}, \quad E_{\xi_{i}} = T_{i} \cdots T_{N - 1} (E_{N}), \quad
i = 1, \cdots, N.
$$
For $i = N$ we mean that $E_{\xi_{N}} = E_{N}$.
In the following we will use many times the identity proven below. It is a replacement for the usual identity involving iterated commutators.

\begin{lemma}
\label{lem:recursion-roots}
For $i = 1, \cdots, N - 1$ we have
$$E_{\xi_{i}} = - E_{i} E_{\xi_{i + 1}} + q^{-1}E_{\xi_{i + 1}} E_{i}.$$
\end{lemma}
\begin{proof}
The result is true for $i = N - 1$ since
\[
E_{\xi_{N-1}}=T_{N-1}(E_{N})=-E_{N-1}E_{N}+q^{-1}E_{N}E_{N-1}.
\]
Now suppose it holds for $i + 1$. Then we have
\[
E_{\xi_{i}}
= T_{i} T_{i + 1} \cdots T_{N - 1}(E_{N})
= T_{i}(E_{\xi_{i + 1}})
= T_{i}(-E_{i + 1} E_{\xi_{i + 2}} + q^{-1}E_{\xi_{i + 2}} E_{i + 1}).
\]
Since $E_{\xi_{i + 2}}$ contains only $E_{i + 2}, \cdots, E_{N}$ we have
that $T_{i}(E_{\xi_{i + 2}}) = E_{\xi_{i + 2}}$. Then
\[
\begin{split}
E_{\xi_{i}} & = -T_{i}(E_{i + 1}) E_{\xi_{i + 2}} + q^{-1} E_{\xi_{i + 2}} T_{i}(E_{i + 1})\\
 & = -(-E_{i} E_{i + 1} + q^{-1} E_{i + 1}E_{i}) E_{\xi_{i + 2}} + q^{-1}E_{\xi_{i + 2}} (-E_{i} E_{i + 1} + q^{-1} E_{i + 1} E_{i}).
\end{split}
\]
By the previous argument $E_{i}$ commutes with $E_{\xi_{i + 2}}$.
Therefore
\[
\begin{split}
E_{\xi_{i}} & = -E_{i} (-E_{i + 1} E_{\xi_{i + 2}} + q^{-1} E_{\xi_{i + 2}} E_{i + 1}) + q^{-1}(-E_{i + 1} E_{\xi_{i + 2}} + q^{-1}E_{\xi_{i + 2}} E_{i + 1}) E_{i}\\
 & = -E_{i} E_{\xi_{i + 1}} + q^{-1} E_{\xi_{i + 1}} E_{k}. \qedhere
\end{split}
\]
\end{proof}

It is also clear from this relation that the quantum root vector $E_{\xi_{i}}$ has weight $\xi_{i}$.

\subsection{Commutation relations}

We will derive here commutation relations between the quantum root vectors $\{E_{\xi_{i}}\}_{i = 1}^N$ and their $*$-counterparts. We will only be interested in commutation relations modulo terms in $U_{q}(\mathfrak{l})$, for which we adopt the following notation.

\begin{notation}
For $X, Y \in U_{q}(\mathfrak{g})$ we write $X \sim Y$ if we have $X = Y + Z$ for some $Z \in U_{q}(\mathfrak{l})$.
\end{notation}

We start by proving two lemmata.

\begin{lemma}
\label{lem:first-lem-comm}
Let $i < N$. Then we have
\[
E_{\xi_{i}}^{*}E_{i}-qE_{i}E_{\xi_{i}}^{*}=-q^{-1}E_{\xi_{i+1}}^{*}.
\]
\end{lemma}
\begin{proof}
We write $E_{\xi_{i}}^{*} E_{i} = E_{\xi_{i}}^{*} F_{i}^{*} K_{i} = (F_{i} E_{\xi_{i}})^{*} K_{i}$.
Then using \cref{lem:recursion-roots} we compute
\[
\begin{split}
[F_{i}, E_{\xi_{i}}]
& = [F_{i}, -E_{i} E_{\xi_{i + 1}} + q^{-1} E_{\xi_{i+1}} E_{i}]\\
& = \frac{K_{i} - K_{i}^{-1}}{q - q^{-1}}E_{\xi_{i + 1}} - q^{-1} E_{\xi_{i + 1}}\frac{K_{i} - K_{i}^{-1}}{q - q^{-1}}.
\end{split}
\]
Using the relations for the generators we easily get $E_{\xi_{i + 1}} K_{i} = q K_{i} E_{\xi_{i + 1}}$. Hence
\[
[F_{i}, E_{\xi_{i}}]
= \left(\frac{K_{i} - K_{i}^{-1}}{q - q^{-1}} - q^{-1} \frac{q K_{i} - q^{-1} K_{i}^{-1}}{q - q^{-1}}\right) E_{\xi_{i + 1}}
= -q^{-1} K_{i}^{-1} E_{\xi_{i + 1}}.
\]
Plugging this into $E_{\xi_{i}}^{*} E_{i} = (F_{i} E_{\xi_{i}})^{*} K_{i}$ we get
\[
\begin{split}
E_{\xi_{i}}^{*} E_{i}
& =(E_{\xi_{i}} F_{i})^{*} K_{i} + (-q^{-1} K_{i}^{-1} E_{\xi_{i + 1}})^{*} K_{i}\\
& =(K_{i} E_{\xi_{i}} F_{i})^{*} - q^{-1}E_{\xi_{i + 1}}^{*}.
\end{split}
\]
Then using again $K_{i} E_{\xi_{i + 1}} = q^{-1}K_{i} E_{\xi_{i + 1}}$ and $K_{i} E_{i} = q^{2}E_{i} K_{i}$ we get
\[
E_{\xi_{i}}^{*}E_{i}
= q(E_{\xi_{i}} E_{i}^{*})^{*} - q^{-1} E_{\xi_{i + 1}}^{*}
= q E_{i} E_{\xi_{i}}^{*} - q^{-1} E_{\xi_{i + 1}}^{*}.
\qedhere
\]
\end{proof}

\begin{lemma}
\label{lem:comm-i-ip1}
Let $i < N$. Then we have $E_{\xi_{i}}^{*} E_{\xi_{i + 1}} \sim q E_{\xi_{i + 1}} E_{\xi_{i}}^{*}$.
\end{lemma}

\begin{proof}
Consider first the case $i = N - 1$. We have
\[
E_{\xi_{N - 1}}^{*} E_{\xi_{N}}
= (-E_{N}^{*} E_{N - 1}^{*} + q^{-1} E_{N - 1}^{*} E_{N}^{*}) E_{N}.
\]
First of all notice that we have
\[
E_{N - 1}^{*} E_{N} = K_{N - 1} F_{N - 1} E_{N}= q^{-1} E_{N} K_{N - 1} F_{N - 1}
= q^{-1} E_{N} E_{N - 1}^{*}.
\]
Secondly we need the identity
\[
\begin{split}
E_{N}^{*} E_{N}
& = K_{N} F_{N} E_{N}
= K_{N} E_{N} F_{N} - K_{N} \frac{K_{N} - K_{N}^{-1}}{q - q^{-1}}\\
& \sim q^{2} E_{N} K_{N} F_{N}
= q^{2} E_{N} E_{N}^{*}.
\end{split}
\]
Using these relations we immediately obtain
\[
E_{\xi_{N - 1}}^{*} E_{\xi_{N}}
\sim q E_{N} (-E_{N}^{*} E_{N - 1}^{*} + q^{-1} E_{N - 1}^{*} E_{N}^{*})
= q E_{\xi_{N}} E_{\xi_{N - 1}}^{*}.
\]

Now we proceed by induction. We use the identity $E_{\xi_{i + 1}}^{*} E_{i + 1} = q E_{i + 1}E_{\xi_{i + 1}}^{*} - q^{-1} E_{\xi_{i + 2}}^{*}$ from \cref{lem:first-lem-comm} and the induction hypothesis $E_{\xi_{i + 1}}^{*} E_{\xi_{i + 2}} \sim q E_{\xi_{i + 2}} E_{\xi_{i + 1}}^{*}$. We obtain 
\[
\begin{split}
E_{\xi_{i + 1}}^{*} E_{\xi_{i + 1}}
& = E_{\xi_{i + 1}}^{*} (-E_{i + 1} E_{\xi_{i + 2}} + q^{-1} E_{\xi_{i + 2}} E_{i + 1})\\
& \sim q^{2} (-E_{i + 1} E_{\xi_{i + 2}} + q^{-1} E_{\xi_{i + 2}} E_{i + 1}) E_{\xi_{i + 1}}^{*}\\
& -q^{-1} (-E_{\xi_{i + 2}}^{*} E_{\xi_{i + 2}} + E_{\xi_{i + 2}} E_{\xi_{i + 2}}^{*}).
\end{split}
\]
This can be recast in the form
\begin{equation}
\label{eq:comm-ii}
E_{\xi_{i + 1}}^{*} E_{\xi_{i + 1}} \sim q^{2} E_{\xi_{i + 1}} E_{\xi_{i + 1}}^{*} + q^{-1}[E_{\xi_{i + 2}}^{*}, E_{\xi_{i + 2}}].
\end{equation}
We will use this result to rewrite the expression
\[
E_{\xi_{i}}^{*} E_{\xi_{i + 1}}
= (-E_{\xi_{i + 1}}^{*} E_{i}^{*} + q^{-1} E_{i}^{*} E_{\xi_{i + 1}}^{*}) E_{\xi_{i + 1}}.
\]
The relation $E_{i}^{*} E_{\xi_{i + 1}} = q^{-1} E_{\xi_{i + 1}} E_{i}^{*}$ is straightforward to check, using the relations for the generators. Using it together with equation \eqref{eq:comm-ii} we obtain
\[
\begin{split}
E_{\xi_{i}}^{*} E_{\xi_{i + 1}}
& \sim q E_{\xi_{i + 1}} (- E_{\xi_{i + 1}}^{*} E_{i}^{*} + q^{-1} E_{i}^{*} E_{\xi_{i + 1}}^{*})\\
& - q^{-2} [E_{\xi_{i + 2}}^{*}, E_{\xi_{i + 2}}] E_{i}^{*} + q^{-2} E_{i}^{*}[E_{\xi_{i + 2}}^{*}, E_{\xi_{i + 2}}].
\end{split}
\]
Finally observe that $E_{i}^{*} = K_{i} F_{i}$ commutes with $E_{\xi_{i + 2}}$ and $E_{\xi_{i + 2}}^{*}$. Therefore the last two terms cancel out and we conclude that
\[
E_{\xi_{i}}^{*} E_{\xi_{i + 1}}
\sim q E_{\xi_{i + 1}} (- E_{\xi_{i + 1}}^{*} E_{i}^{*} + q^{-1} E_{i}^{*} E_{\xi_{i + 1}}^{*})
= q E_{\xi_{i + 1}} E_{\xi_{i}}^{*}. \qedhere
\]
\end{proof}

At this point we can easily obtain the commutation relations for $i \neq j$.

\begin{proposition}
Let $i \neq j$. Then we have $E_{\xi_{i}}^{*} E_{\xi_{j}} \sim q E_{\xi_{j}} E_{\xi_{i}}^{*}$.\end{proposition}

\begin{proof}
First we consider the case $i < j$. For $i = j - 1$ the statement follows from \cref{lem:comm-i-ip1}. We proceed by induction, that is we assume that $E_{\xi_{i + 1}}^{*} E_{\xi_{j}} \sim q E_{\xi_{j}} E_{\xi_{i + 1}}^{*}$ for $ i < j - 1$. Using the usual recursion relation for $E_{\xi_{i}}$ we get 
\[
E_{\xi_{i}}^{*} E_{\xi_{j}}
= (-E_{\xi_{i + 1}}^{*} E_{i}^{*} + q^{-1} E_{i}^{*} E_{\xi_{i + 1}}^{*}) E_{\xi_{j}}.
\]
We have that $[E_{i}^{*}, E_{\xi_{j}}] = 0$, since $E_{\xi_{j}}$ contains the generators $E_{j}, \cdots, E_{N}$ and $i < j - 1$. Using this fact together with $E_{\xi_{i + 1}}^{*} E_{\xi_{j}} \sim q E_{\xi_{j}} E_{\xi_{i + 1}}^{*}$ we obtain
\[
E_{\xi_{i}}^{*} E_{\xi_{j}}
\sim q E_{\xi_{j}} ( - E_{\xi_{i + 1}}^{*} E_{i}^{*} + q^{-1} E_{i}^{*} E_{\xi_{i + 1}}^{*}) = q E_{\xi_{j}} E_{\xi_{i}}^{*}.
\]

Now we consider the case $i > j$. First observe that the operation $*$ leaves $U_{q}(\mathfrak{l})$ invariant, since it is a Hopf $*$-subalgebra of $U_{q}(\mathfrak{g})$.
This implies that if $A \sim B$ then we also have $A^{*} \sim B^{*}$.
Then applying $*$ to the relation $E_{\xi_{a}}^{*} E_{\xi_{b}} \sim q E_{\xi_{b}} E_{\xi_{a}}^{*}$ with $a < b$ we get $E_{\xi_{b}}^{*} E_{\xi_{a}} \sim q E_{\xi_{a}} E_{\xi_{b}}^{*}$.
The result follows upon setting $i = b$ and $j = a$.
\end{proof}

In the following we will need commutation relations for the elements $S^{-1}(E_{\xi_{i}})$. It is convenient to introduce a special notation for them.

\begin{notation}
For $i = 1, \cdots, N$ we set $\mathcal{E}_{i} = S^{-1}(E_{\xi_{i}})$.
\end{notation}

We can now easily get commutation relations between $\mathcal{E}_{i}$ and $\mathcal{E}_{j}^{*}$ with $i \neq j$ from the previous lemma, using various properties of the antipode.

\begin{corollary}
\label{cor:comm-inj}
Let $i\neq j$. Then we have $\mathcal{E}_{i} \mathcal{E}_{j}^{*} \sim q^{-1} \mathcal{E}_{j}^{*} \mathcal{E}_{i}$.
\end{corollary}

\begin{proof}
First of all observe that $S$ and $S^{-1}$ leave $U_{q}(\mathfrak{l})$ invariant. Then if $A\sim B$ we also have $S^{-1}(A)\sim S^{-1}(B)$.
Therefore we obtain
\[
S^{-1}(E_{\xi_{j}})S^{-1}(E_{\xi_{i}}^{*})\sim qS^{-1}(E_{\xi_{i}}^{*})S^{-1}(E_{\xi_{j}}).
\]
Using the general relation $S^{-1} \circ * = * \circ S$ we rewrite it as
\[
S(E_{\xi_{i}})^{*}S^{-1}(E_{\xi_{j}})\sim q^{-1}S^{-1}(E_{\xi_{j}})S(E_{\xi_{i}})^{*}.
\]
From the definition of the antipode we get $S(E_{i}) = q^{-2} S^{-1}(E_{i})$ for all $i$.
Since $E_{\xi_{i}}$ contains the generators $E_{i}, \cdots, E_{N}$ we see that $S(E_{\xi_{i}}) = q^{-2(N - i + 1)}S^{-1}(E_{\xi_{i}})$. Alternatively one can use the general fact that $S^{2}(E_{\xi_{i}}) = q^{(2\rho,\xi_{i})}E_{\xi_{i}}$, where $\rho$ is the half-sum of the positive roots.
In any case, using one of these identities and simplifying we get
\[
S^{-1}(E_{\xi_{i}})^{*}S^{-1}(E_{\xi_{j}})\sim q^{-1}S^{-1}(E_{\xi_{j}})S^{-1}(E_{\xi_{i}})^{*}. \qedhere
\]
\end{proof}

Next we will consider the case $i = j$, where the relations turn out to be more complicated. Nevertheless, we will see that this is exactly what we need to obtain a simple formula for the square of the Dolbeault–Dirac operator.
The next result should be compared with the case $i = j$ of the relations in the quantum Clifford algebra, see \cref{prop:cliff-iej}.

\begin{proposition}
Let $i \leq N$. Then we have
\[
E_{\xi_{i}} E_{\xi_{i}}^{*} - q^{-2} E_{\xi_{i}}^{*} E_{\xi_{i}} \sim -q^{-1}(q - q^{-1}) \sum_{k = i + 1}^{N} q^{3(i - k)} E_{\xi_{k}}^{*} E_{\xi_{k}}.
\]
\end{proposition}

\begin{proof}
First of all observe that for $i = N$ this is $E_{\xi_{N}} E_{\xi_{N}}^{*} - q^{-2} E_{\xi_{N}}^{*} E_{\xi_{N}} \sim 0$, which we have proven in the course of \cref{lem:comm-i-ip1}. Also recall the relation \eqref{eq:comm-ii}, which reads
\[
E_{\xi_{i}}^{*} E_{\xi_{i}} - q^{2} E_{\xi_{i}}E_{\xi_{i}}^{*} \sim q^{-1} [E_{\xi_{i + 1}}^{*}, E_{\xi_{i + 1}}].
\]
We can rewrite it in the following form
\[
E_{\xi_{i}} E_{\xi_{i}}^{*} - q^{-2} E_{\xi_{i}}^{*} E_{\xi_{i}}
\sim q^{-3} (E_{\xi_{i + 1}} E_{\xi_{i + 1}}^{*} - q^{-2} E_{\xi_{i + 1}}^{*} E_{\xi_{i + 1}}) - q^{-4} (q - q^{-1}) E_{\xi_{i + 1}}^{*} E_{\xi_{i + 1}}.
\]
Now we proceed by induction. We plug the formula for $E_{\xi_{i + 1}} E_{\xi_{i + 1}}^{*} - q^{-2} E_{\xi_{i + 1}}^{*} E_{\xi_{i + 1}}$ into the previous identity. Then upon relabeling the sum we get
\[
\begin{split}
E_{\xi_{i}} E_{\xi_{i}}^{*} - q^{-2} E_{\xi_{i}}^{*} E_{\xi_{i}}
& = -q^{-4} (q - q^{-1}) \sum_{k = i + 2}^{N} q^{3(i + 1 - k)} E_{\xi_{k}}^{*} E_{\xi_{k}}-q^{-4} (q - q^{-1}) E_{\xi_{i + 1}}^{*} E_{\xi_{i + 1}}\\
& = - q^{-1}(q - q^{-1}) \sum_{k = i + 1}^{N}q^{3(i - k)} E_{\xi_{k}}^{*} E_{\xi_{k}}. \qedhere
\end{split}
\]
\end{proof}

\begin{corollary}
\label{cor:comm-iej}
Let $i \leq N$. Then we have
\[
\mathcal{E}_{i}^{*} \mathcal{E}_{i} - q^{-2} \mathcal{E}_{i} \mathcal{E}_{i}^{*}
\sim - q^{-1} (q - q^{-1}) \sum_{k = i + 1}^{N} q^{i - k} \mathcal{E}_{k} \mathcal{E}_{k}^{*}.
\]
\end{corollary}

\begin{proof}
Applying $S^{-1}$ and using $S^{-1} \circ * = * \circ S$ in the previous lemma we get
\[
S(E_{\xi_{i}})^{*} S^{-1}(E_{\xi_{i}}) - q^{-2} S^{-1} (E_{\xi_{i}}) S(E_{\xi_{i}})^{*}
\sim - q^{-1}(q - q^{-1}) \sum_{k = i + 1}^{N} q^{3(i - k)} S^{-1}(E_{\xi_{k}}) S(E_{\xi_{k}})^{*}.
\]
Now recall the relation $S(E_{\xi_{i}}) = q^{-2(N - i + 1)}S^{-1}(E_{\xi_{i}})$, derived during the proof of \cref{cor:comm-inj}. Plugging this in and simplifying we obtain the claimed result.
\end{proof}

\subsection{Adjoint action}
By \cite[Theorem 5.6]{zwi} the algebra generated by the quantum root vectors $\{E_{\xi_{i}}\}_{i = 1}^N$, called the \emph{twisted quantum Schubert cell}, can be identified as a graded $U_{q}(\mathfrak{l})$-module with the quantum symmetric algebra $S_{q}(\mathfrak{u}_{+})$.
We have previously introduced an orthonormal basis $\{e_{i}\}_{i = 1}^N$ with respect a Hermitian inner product on $\mathfrak{u}_{+}$.
In this subsection we will give the relation between these two bases.
In order to do that we need to compute the adjoint action of $U_{q}(\mathfrak{l})$ on the quantum root vectors.
Recall that the (left) adjoint action is defined by $X \triangleright Y = X_{(1)} Y S(X_{(2)})$, where we use the standard Sweedler notation.

\begin{lemma}
\label{lem:adj-act}
Let $i \leq N$ and $j < N$. Then we have
\[
E_{j} \triangleright E_{\xi_{i}} = -\delta_{j,i - 1} E_{\xi_{i - 1}}, \quad
F_{j} \triangleright E_{\xi_{i}} = -\delta_{j,i} E_{\xi_{i + 1}}.
\]
\end{lemma}
\begin{proof}
It follows from general arguments that the subspace generated by $\{ E_{\xi_{i}} \}_{i = 1}^N$ is invariant under the adjoint action of $U_{q}(\mathfrak{l})$, see \cite[Theorem 5.6]{zwi}. Therefore
$$
E_{j} \triangleright E_{\xi_{i}} = \sum_{k = 1}^{N} c_{i j}^{k} E_{\xi_{k}},
\quad c_{i j}^{k} \in \mathbb{C}.
$$
Acting with an element $K_{\lambda}$ and using the fact that $K_{\lambda} \triangleright E_{\xi_{k}} = q^{(\lambda, \xi_{k})} E_{\xi_{k}}$ we get
$$
K_{\lambda} \triangleright (E_{j} \triangleright E_{\xi_{i}}) = \sum_{k = 1}^{N} c_{i j}^{k} q^{(\lambda, \xi_{k})} E_{\xi_{k}}.
$$
On the other hand we have $E_{j} \triangleright E_{\xi_{i}} = E_{j} E_{\xi_{i}} - K_{j} E_{\xi_{i}} K_{j}^{-1} E_{j}$. Acting with $K_{\lambda}$ on the right-hand side of this expression and using $K_{\lambda} \triangleright (X Y) = (K_{\lambda} \triangleright X) (K_{\lambda} \triangleright Y)$ we get
$$
K_{\lambda} \triangleright (E_{j} \triangleright E_{\xi_{i}})
= q^{(\lambda, \xi_{i})} q^{(\lambda, \alpha_{j})} E_{j} \triangleright E_{\xi_{i}}
$$
Therefore, using the linear independence of the quantum root vectors, we see that the term $E_{\xi_{k}}$ appears in the sum inside $E_{j} \triangleright E_{\xi_{i}}$ if and only if $\xi_{k} = \xi_{i} + \alpha_{j}$.

In our case the radical roots can be written in terms of the simple roots as $\xi_{i} = \sum_{a = i}^{N} \alpha_{a}$. Therefore we get an equality only in the case $j = k = i - 1$. Then we obtain
$$
E_{i - 1} \triangleright E_{\xi_{i}} = c_{i, i - 1}^{i - 1} E_{\xi_{i - 1}}, \quad
E_{j} \triangleright E_{\xi_{i}} = 0, \quad j \neq i - 1.
$$
The missing coefficient can be easily found as follows
\[
\begin{split}
E_{i - 1} \triangleright E_{\xi_{i}}
& = E_{i - 1} E_{\xi_{i}} - K_{i - 1} E_{\xi_{i}} K_{i - 1}^{-1} E_{i - 1}\\
& = E_{i - 1} E_{\xi_{i}} - q^{-1} E_{\xi_{i}}E_{i - 1}
= - E_{\xi_{i - 1}}.
\end{split}
\]

By a similar argument one finds the result also for $F_{j} \triangleright E_{\xi_{i}}$.
\end{proof}

For the next proposition we use the notation $V = \mathrm{span} \{ E_{\xi_{1}}, \cdots, E_{\xi_{N}} \}$.

\begin{proposition}
\label{prop:orthonorm-basis}
Let $\{ e_{i} \}_{i = 1}^N$ be the weight basis of $\mathfrak{u}_{+}$ introduced previously.
Then we can identify this basis with the rescaled quantum root vectors $\{ \tilde{e}_{i} = q^{-i/2} E_{\xi_{i}} \}_{i = 1}^N$ of $V$.
\end{proposition}
\begin{proof}
We will define a Hermitian inner product $(\cdot, \cdot): V \otimes V \to \mathbb{C}$ such that $(v, X \triangleright w) = (X^{*} \triangleright v, w)$ for all $X \in U_{q}(\mathfrak{l})$ and $v, w \in V$.
To this end, we rescale the quantum root vectors as $\tilde{e}_{i} = q^{-i/2} E_{\xi_{i}} \in U_{q}(\mathfrak{g})$.
By the previous lemma we have
$$
E_{j} \triangleright \tilde{e}_{i} = - \delta_{j, i - 1} q^{-1/2} \tilde{e}_{i - 1}, \quad
F_{j} \triangleright \tilde{e}_{i} = - \delta_{j,i} q^{1/2} \tilde{e}_{i + 1}.
$$
Now define $(\tilde{e}_{i}, \tilde{e}_{j}) = \delta_{ij}$. We will show that it satisfies the claimed property.

It is clear that $(\tilde{e}_{i}, K_{\lambda} \triangleright \tilde{e}_{j}) = (K_{\lambda}^{*} \triangleright \tilde{e}_{i}, \tilde{e}_{j})$, since $\{ \tilde{e}_{i} \}$ is a weight basis and $K_{\lambda}^{*} = K_{\lambda}$.
Now consider the action of the generators $E_{j}$. The only non-zero inner product is $(\tilde{e}_{i - 1}, E_{i - 1} \triangleright \tilde{e}_{i}) = - q^{-1/2}$.
On the other hand, using $E_{j}^{*} = K_{j} F_{j}$, we easily compute
\[
\begin{split}
(E_{i - 1}^{*} \triangleright \tilde{e}_{i - 1}, \tilde{e}_{i}) 
& =(K_{i - 1} F_{i - 1} \triangleright \tilde{e}_{i - 1}, e_{i})
= - q^{1/2} (K_{i - 1} \triangleright \tilde{e}_{i}, e_{i})\\
& = - q^{1/2} q^{-1} (\tilde{e}_{i}, \tilde{e}_{i}) = - q^{-1/2}.
\end{split}
\]
Therefore the claim is true for $E_{j}$. A similar computation show that this is the case also for the generators $F_{j}$. Hence $(\cdot, \cdot)$ satisfies the claimed property.

Now recall that we have set $(e_{i}, e_{j}) = \delta_{ij}$ for the basis $\{ e_{i} \}_{i = 1}^N$ of $\mathfrak{u}_{+}$. Since $\{ e_{i} \}_{i = 1}^N$ and $\{ \tilde{e}_{i} \}_{i = 1}^N$ are weight bases, the claimed identification follows.
\end{proof}

\section{Dolbeault–Dirac operator}
\label{sec:dolbeault}

In this section we will compute the square of the Dolbeault–Dirac operator, up to terms which live in the quantized Levi factor $\Uql$.
It will take a particularly simple form only for some choices corresponding to the dual pairings and Hermitian inner products. We will briefly discuss the meaning of these choices.
As a byproduct of this formula, we will prove that the Dolbeault–Dirac operator has compact resolvent, similarly to \cite{dd-proj}.

\subsection{Computation of $D^{2}$}

First we recall its definition, as in \cite[Definition 5.7]{qflag2}.
Fix a weight basis $\{x_{i}\}_{i = 1}^N$ of the $U_{q}(\mathfrak{l})$-module $\mathfrak{u}_{+}$, in such a way that $x_{i}$ corresponds to $E_{\xi_{i}}$ under the previously mentioned isomorphism, see \cite[Convention 4.1]{qflag2}.
Let $\{y_{i}\}_{i = 1}^N$ be the dual basis of $\mathfrak{u}_{-}$ with respect to the pairing.
Also recall that we write $\mathcal{E}_{i} = S^{-1}(E_{\xi_{i}})$.

\begin{definition}
With the notation as above, we set
$$
\eth = \sum_{i = 1}^{N} \mathcal{E}_{i} \otimes \gamma_{-}(y_{i}).
$$
Then the \emph{Dolbeault–Dirac operator} is defined as $D = \eth + \eth^{*}$.
\end{definition}

To ease the notation we abbreviate $\gamma_{-}(y_{i}) = \gamma_{i}$ in the next lemma.

\begin{lemma}
\label{lem:square-d}
We have $D^{2} = D^{2}_{D} + D^{2}_{O}$, where
\[
\begin{split}
D_{D}^{2} & \sim \sum_{i = 1}^{N} \mathcal{E}_{i} \mathcal{E}_{i}^{*} \otimes \left(\gamma_{i} \gamma_{i}^{*} + q^{-2} \gamma_{i}^{*} \gamma_{i} - q^{-1} (q - q^{-1}) \sum_{j = 1}^{i - 1} q^{j - i} \gamma_{j}^{*} \gamma_{j}\right), \\
D_{O}^{2} & \sim \sum_{i \neq j} \mathcal{E}_{i} \mathcal{E}_{j}^{*} \otimes (\gamma_{i} \gamma_{j}^{*} + q^{-1} \gamma_{j}^{*} \gamma_{i}).
\end{split}
\]
\end{lemma}

\begin{proof}
We know from \cite[Theorem 1.1]{qflag2} that $\eth^{2} = 0$. Therefore we get
$$
D^{2} = \sum_{i, j = 1}^{N} \mathcal{E}_{i} \mathcal{E}_{i}^{*} \otimes \gamma_{i} \gamma_{j}^{*} + \sum_{i, j = 1}^{N} \mathcal{E}_{i}^{*} \mathcal{E}_{i} \otimes \gamma_{i}^{*} \gamma_{j}.
$$
We divide $D^{2}$ into the two terms $D^{2}_{D}$ and $D^{2}_{O}$, depending on whether $i = j$ or $i \neq j$ (the subscripts denote respectively diagonal and off-diagonal).
We will rewrite these terms using the commutation relations obtained previously.
We start with $D_{O}^{2}$. From \cref{cor:comm-inj} we have the relation $\mathcal{E}_{i}^{*} \mathcal{E}_{j} \sim q^{-1} \mathcal{E}_{j} \mathcal{E}_{i}^{*}$ for $i \neq j$. Plugging this in and relabeling we obtain
\[
D_{O}^{2} \sim \sum_{i \neq j} \mathcal{E}_{i} \mathcal{E}_{j}^{*} \otimes (\gamma_{i} \gamma_{j}^{*} + q^{-1} \gamma_{j}^{*} \gamma_{i}).
\]
The term $D_{D}^{2}$ is more complicated. From \cref{cor:comm-iej} we have
\[
\mathcal{E}_{i}^{*} \mathcal{E}_{i} \sim q^{-2} \mathcal{E}_{i} \mathcal{E}_{i}^{*} -q^{-1} (q - q^{-1}) \sum_{j = i + 1}^{N} q^{i - j}\mathcal{E}_{j} \mathcal{E}_{j}^{*}.
\]
Plugging this in we get
\[
D^{2}_{D} \sim \sum_{i = 1}^{N} \mathcal{E}_{i} \mathcal{E}_{i}^{*} \otimes (\gamma_{i} \gamma_{i}^{*} + q^{-2} \gamma_{i}^{*} \gamma_{i}) - q^{-1} (q - q^{-1})\sum_{i = 1}^{N} \sum_{j = i + 1}^{N} q^{i - j} \mathcal{E}_{j} \mathcal{E}_{j}^{*} \otimes \gamma_{i}^{*} \gamma_{i}.
\]
We can invert the two sums in the last term using the general relation
\[
\sum_{i=1}^{N}\sum_{j=i+1}^{N}f(i,j)=\sum_{i=1}^{N}\sum_{j=1}^{i-1}f(j,i).
\]
Therefore we obtain the expression
\[
D^{2}_{D} \sim \sum_{i = 1}^{N} \mathcal{E}_{i} \mathcal{E}_{i}^{*} \otimes \left(\gamma_{i} \gamma_{i}^{*} + q^{-2} \gamma_{i}^{*} \gamma_{i} - q^{-1} (q - q^{-1}) \sum_{j = 1}^{i - 1} q^{j - i} \gamma_{j}^{*} \gamma_{j}\right). \qedhere
\]
\end{proof}

To proceed we will need the commutation relations in the quantum Clifford algebra. These depend on the choice of the pairings and the Hermitian inner products, that is on the numbers $\{\lambda_{k}\}_{k = 0}^N$ and $\{\lambda^{\prime}_{k}\}_{k = 0}^N$ introduced in \cref{not:dual-pair} and \cref{not:herm-prod}.
More precisely they will appear only through ratios of the parameters $c_{k} = |\lambda_{k}|^{2} / \lambda_{k}^{\prime}$, as defined in \cref{cor:cliff-rescaling}.
The reason for this is that $D^{2}$ contains only certain quadratic elements in the generators.

Before getting into the computation we observe a few things. Recall that $\{ x_{i} \}_{i = 1}^N$ denotes the basis of $\mathfrak{u}_{+}$ which we identify with the quantum root vectors $\{ E_{\xi_{i}} \}_{i = 1}^N$. Denote by $\{ y_{i} \}_{i = 1}^N$ the dual basis of $\mathfrak{u}_{-}$ with respect to the dual pairing. These appear in the definition of the Dolbeault–Dirac operator.
In \cref{prop:orthonorm-basis} we have seen that the orthonormal basis $\{ e_{i} \}_{i = 1}^N$ can be related to the quantum root vectors by $e_{i} = q^{-i/2} E_{\xi_{i}}$. Then, denoting by $\{ f_{i} \}_{i = 1}^N$ the dual basis, we find $f_{i} = q^{i/2} y_{i}$. Hence we have the relations $\gamma_{-}(y_{i}) = q^{-i/2} \gamma_{-}(f_{i})$.

In the following we will use the short-hand notation $\gamma_{i} = \gamma_{-}(f_{i})$.

\begin{lemma}
We have $D_{O}^{2} \sim 0$ if and only if $\frac{c_{k + 1}}{c_k} = \frac{c_k}{c_{k - 1}} q^{-2}$ holds for $1 \leq k \leq N - 1$.
\end{lemma}

\begin{proof}
We have previously shown that
$$
D_{O}^{2} \sim \sum_{i \neq j} \mathcal{E}_{i} \mathcal{E}_{j}^{*} \otimes (\gamma_{-}(y_{i}) \gamma_{-}(y_{j})^{*} + q^{-1} \gamma_{-}(y_{j})^{*} \gamma_{-}(y_{i})).
$$
First we replace $\gamma_{-}(y_{i})$ with $q^{-i/2} \gamma_{i}$, so that
$$
D_{O}^{2} \sim \sum_{i \neq j} q^{-(i + j)/2} \mathcal{E}_{i} \mathcal{E}_{j}^{*} \otimes (\gamma_{i} \gamma_{j}^{*} + q^{-1} \gamma_{j}^{*} \gamma_{i}).
$$
Now we will concentrate on the second factor, that is $\gamma_i \gamma_j^* + q^{-1} \gamma_j^* \gamma_i$. Since $i \neq j$ it is clear that this is zero when acting on elements of degree $0$ or $N$. We consider its action on elements of degree $1 \leq k \leq N - 1$. Using the identities in \cref{cor:cliff-rescaling} we get
$$
\gamma_{i} \gamma_{j}^{*} + q^{-1} \gamma_{j}^{*} \gamma_{i}
= \frac{c_{k + 1}}{c_{k}} \intm{i} \extm{j} + \frac{c_{k}}{c_{k - 1}} q^{-1} \extm{j} \intm{i}.
$$
Recall that we have the identity $\extm{i} \intm{j} = - q^{-1} \intm{j} \extm{i}$ for $i \neq j$. Plugging this in we get
$$
\gamma_{i} \gamma_{j}^{*} + q^{-1} \gamma_{j}^{*} \gamma_{i}
= \left( \frac{c_{k + 1}}{c_{k}} - \frac{c_{k}}{c_{k - 1}} q^{-2} \right) \intm{i} \extm{j}.
$$
Therefore we see that this term cancels when the claimed condition holds.
\end{proof}

Now we will show that, under the same assumption on the coefficients $\{c_{k}\}_{k = 0}^N$, also the diagonal part of $D^{2}$ simplifies.
During the computation we will introduce a certain element $C \in U_q(\mathfrak{sl}_{N + 1})$.
This is related to a central element $\mathcal{C}_q \in U_q(\mathfrak{sl}_{N + 1})$ which reduces to the quadratic Casimir in the classical limit, as we will show in \cref{sec:casimir}.
Therefore the next result can be seen as a generalization of the classical formula for $D^{2}$.

\begin{theorem}
\label{thm:dsquare}
Suppose that the condition $\frac{c_{k + 1}}{c_{k}} = \frac{c_{k}}{c_{k - 1}} q^{-2}$ holds for $1 \leq k \leq N - 1$. Then we have $D^{2} \sim C \otimes T$, where $C = \sum_{i = 1}^{N} q^{-i} \mathcal{E}_{i} \mathcal{E}_{i}^{*} \in U_q(\mathfrak{sl}_{N + 1})$ and $T \in \mathrm{End}(\Lambdaq)$ is defined by
\[
T \restriction_{\Lambdak} = \begin{cases}
\frac{c_1}{c_0} q^{-2k} & k \leq 2, \\
q^{-2k} & k > 2.
\end{cases}
\]
\end{theorem}

\begin{proof}
We have just shown that $D^{2}_{O} \sim 0$.
Plugging $\gamma_{-}(y_{i}) = q^{-i/2} \gamma_{i}$ into \cref{lem:square-d} we obtain
\[
D^{2}_{D} \sim \sum_{i = 1}^{N} q^{-i} \mathcal{E}_{i} \mathcal{E}_{i}^{*} \otimes \left(\gamma_{i} \gamma_{i}^{*} + q^{-2} \gamma_{i}^{*} \gamma_{i} - q^{-1} (q - q^{-1}) \sum_{j = 1}^{i - 1} \gamma_{j}^{*} \gamma_{j}\right).
\]
We concentrate on the terms appearing in the second factor, which we denote by $A_{i}$. We act with these operators  on elements of degree $k$. Using the identities in \cref{cor:cliff-rescaling} we find
$$
A_{i} = \frac{c_{k + 1}}{c_{k}} \intm{i} \extm{i} + \frac{c_{k}}{c_{k - 1}} q^{-2} \left( \extm{i} \intm{i} - q (q - q^{-1}) \sum_{j = 1}^{i - 1} \extm{j} \intm{j} \right).
$$
In degree $k = 0$ we have $\intm{i} \extm{i} = \mathrm{id}$ and $\extm{i} \intm{i} = 0$, hence $A_i$ acts as multiplication by $\frac{c_1}{c_0}$.
Now consider the case $k > 0$ and recall that from \cref{prop:cliff-iej} we have the relation
$$
\extm{i} \intm{i} - q(q - q^{-1}) \sum_{j = 1}^{i - 1} \extm{j} \intm{j} + \intm{i} \extm{i} = \mathrm{id}.
$$
Plugging this into the previous equation we get
$$
A_{i} = \left( \frac{c_{k + 1}}{c_{k}} - \frac{c_{k}}{c_{k - 1}} q^{-2} \right) \intm{i} \extm{i} + \frac{c_{k}}{c_{k - 1}} q^{-2} \mathrm{id}.
$$
The first term is zero for $1 \leq k \leq N - 1$ under our assumption on the coefficients $\{c_{k}\}_{k = 0}^N$, hence $A_i$ acts as $\frac{c_k}{c_{k - 1}} q^{-2}$.
Similarly in degree $k = N$ we have $\intm{i} \extm{i} = 0$ and $A_i$ acts as $\frac{c_N}{c_{N - 1}} q^{-2}$.
To summarize these results let us define the operator $T: \Lambdak \to \Lambdak$ by
\[
T \restriction_{\Lambdak} = \begin{cases}
\frac{c_1}{c_0} & k = 0, \\
\frac{c_k}{c_{k - 1}} q^{- 2} & k > 0.
\end{cases}
\]
Armed with this definition we can write
\[
D^{2} \sim \sum_{i = 1}^{N} q^{-i} \mathcal{E}_{i} \mathcal{E}_{i}^{*} \otimes T = C \otimes T.
\]
Finally the expression for $T$ can be simplified.
Consider again the condition on the coefficients $\{c_k\}_{k = 0}^N$, which we rewrite as $c_{k + 1} = \frac{c_k^2}{c_{k - 1}} q^{-2}$ for $1 \leq k \leq N - 1$.
From this we easily find that $c_k = \frac{c_1^2}{c_0} q^{-k(k - 1)}$ for $2 \leq k \leq N$, while $c_0$ and $c_1$ are left as free parameters.
Using this result we compute the ratios $\frac{c_2}{c_1} = \frac{c_1}{c_0} q^{-2}$ and $\frac{c_k}{c_{k - 1}} = q^{2 - 2k}$ for $k > 2$.
Plugging this back into the expression for $T$ we obtain the form given in the claim.
\end{proof}

Therefore, under the condition $\frac{c_{k + 1}}{c_{k}} = \frac{c_{k}}{c_{k - 1}} q^{-2}$, we obtain a simple form for the square of the abstract Dolbeault–Dirac operator. Below we make some comments on this condition and the role of the operator $T$ appearing in the statement of the theorem.

\begin{remark}
Consider the classical limit of the previous condition, that is $q \to 1$. This gives $\frac{c_{k + 1}}{c_{k}} = \frac{c_{k}}{c_{k - 1}}$.
Plugging this into \cref{cor:cliff-rescaling} we obtain the identity $\gamma_{i}^{*} \gamma_{j} + \gamma_{i} \gamma_{j}^{*} = \frac{c_{k}}{c_{k - 1}} (\extm{i} \intm{j} + \intm{i} \extm{j})$, when acting in degree $k$.
Using the relations $\extm{i} \intm{j} + \intm{i} \extm{j} = \delta_{ij}$, we find $\gamma_{i}^{*} \gamma_{j} + \gamma_{i} \gamma_{j}^{*} = \delta_{ij} T$.
Therefore the operator $T$ provides a rescaling of these relations in low degrees.
We are free to set $c_{k} = 1$ for all $k$ to recover the usual Clifford algebra relations.
Nevertheless, the condition $\frac{c_{k + 1}}{c_{k}} = \frac{c_{k}}{c_{k - 1}}$ is all we need to get cancellations in the expression for $D^{2}$.
\end{remark}

\begin{remark}
Clearly in the quantum case it is not possible to satisfy $\frac{c_{k + 1}}{c_{k}} = \frac{c_{k}}{c_{k - 1}} q^{-2}$ by setting $c_{k} = 1$ for all $k$. Therefore the classical normalization does not produce a simple result for $D^{2}$.
Observe that the factor $q^{- k (k - 1)}$, which appears in the formula for $c_k$, is essentially the same factor coming from the computation of the dual pairings of \cref{sec:pairing}.
\end{remark}

\subsection{Action on the spinor bundle}

The abstract Dolbeault–Dirac operator $D$ can be made to act on the Hilbert space of square-integrable spinors.
We will follow \cite[Section 6]{qflag} and refer to this paper for the unexplained definitions.
Denote by $\mathbb{C}_q[G]$ the quantized coordinate ring corresponding to $G$. Corresponding to the generalized flag manifold $G / P$ we also have the ring $\mathbb{C}_q[G / P] \subset \mathbb{C}_q[G]$, namely elements which are invariant under the action of the quantized Levi factor $\Uql$.
The corresponding \emph{spinor bundle} is then defined by
\[
\mathcal{S} = \{a \otimes v \in \mathbb{C}_q[G] \otimes \Lambdaq : X \triangleright a \otimes v = a \otimes S(X) \triangleright v, \forall X \in \Uql \}.
\]
More precisely, in the classical situation this linear space corresponds to the space of sections of the spinor bundle.
It can be completed to a Hilbert space using the Haar state on $\mathbb{C}_q[G]$ and the Hermitian inner product on the exterior algebra $\Lambdaq$.
It follows from its definition that the Dolbeault–Dirac operator $D$ is a well-defined operator on this space.

In this setting we obtain a further simplification in our formula for $D^2$.
We will make use of the central element $\mathcal{C}_q \in U_q(\mathfrak{sl}_{N + 1})$, which will be discussed in \cref{sec:casimir}.

\begin{theorem}
\label{thm:dirac-spinor}
We have $D^2 \sim \mathcal{C}_q \otimes \tilde{T}$ as operators on the spinor bundle $\mathcal{S}$, where
\[
\tilde{T} \restriction_{\Lambda_q^k(\mathfrak{u}_{+})} = \begin{cases}
\frac{c_1}{c_0} & k \leq 2, \\
1 & k > 2.
\end{cases}
\]
In particular for $c_0 = c_1$ we have $D^2 \sim \mathcal{C}_q \otimes 1$.
\end{theorem}

\begin{proof}
From \cref{thm:dsquare} we have that $D^2 \sim C \otimes T$. Moreover it is shown in \cref{sec:casimir} that $C \sim \mathcal{C}_q K_{\omega_N}^{-2}$, where $\mathcal{C}_q \in U_q(\mathfrak{sl}_{N + 1})$ is a central element that reduces to the quadratic Casimir in the classical limit.
Here $K_{\omega_N}$ is a central element in $\Uql$, corresponding to the fundamental weight $\omega_N$ (recall that for projective spaces we remove the root $\alpha_N$).

Now for any section $a \otimes v \in \mathcal{S}$ we have $K_\lambda \triangleright a \otimes v = a \otimes K_\lambda^{-1} \triangleright v$, for any Cartan element $K_\lambda$.
Therefore we obtain the relation $D^2 \sim \mathcal{C}_q \otimes K_{\omega_N}^2 T$ as operators on the Hilbert space $\mathcal{S}$.
Next we look at the action of $\tilde{T} = K_{\omega_N}^2 T$ on $\Lambdak$.
Since $K_{\omega_N}$ is a central element in $\Uql$ it acts as a multiple of the identity on $\Uql$-modules.
This number is $q$ on $\mathfrak{u}_+$, see \cite[Section 3.1]{qflag2}, and hence $q^k$ on $\Lambdak$. Observe that for $k = 0$ we have the trivial $\Uql$-module.
Plugging into the expression of $T$ given in \cref{thm:dsquare} we obtain
\[
K_{\omega_N}^2 T \restriction_{\Lambda_q^k(\mathfrak{u}_{+})} = \begin{cases}
\frac{c_1}{c_0} & k \leq 2, \\
1 & k > 2.
\end{cases}
\]
Therefore if we set $c_0 = c_1$ we have that $\tilde{T} = K_{\omega_N}^2 T$ acts as the identity operator.
\end{proof}

\begin{remark}
The condition $c_0 = c_1$ appearing in the above result is a very natural one. In particular it can be obtained by setting $\lambda_0 = \lambda_0^\prime = 1$ and $\lambda_1 = \lambda_1^\prime = 1$.
The conditions $\lambda_0 = \lambda_0^\prime = 1$ in degree zero mean that $\langle 1, 1 \rangle = (1, 1) = 1$, which are standard normalization conditions.
On the other hand the condition $\lambda_1 = 1$ means that $\{f_i\}_{i = 1}^N$ and $\{e_i\}_{i = 1}^N$ are dual bases with respect to the dual pairing, while the condition $\lambda_1^\prime = 1$ means that $\{e_i\}_{i = 1}^N$ is an orthonormal basis with respect to the Hermitian inner product.
\end{remark}

\subsection{Compact resolvent}

It is an important requirement in the theory of spectral triples \cite{con-book} that $D$ should have compact resolvent.
For quantum projective spaces this was proven in \cite{dd-proj}.
We can easily give a new proof using the results obtained here.

\begin{theorem}
The Dolbeault–Dirac operator $D$ has compact resolvent.
\end{theorem}

\begin{proof}
First of all recall the following simple fact: if $S$ is a self-adjoint operator with compact resolvent then the same is true for $S + B$, where $B$ is bounded operator.
In our setting this is useful because, thanks to the equivariance condition defining the spinor bundle $\mathcal{S}$, any element of $U_{q}(\mathfrak{l}) \otimes \mathrm{Cl}_{q}$ acts as a bounded operator, since $\Lambdaq$ is finite-dimensional.

Now we make use of the result that $D^2 \sim \mathcal{C}_q \otimes 1$ from \cref{thm:dirac-spinor}. We could as well consider $D^2 \sim \mathcal{C}_q \otimes \tilde{T}$, but it makes no real difference here.
To prove that $D$ has compact resolvent it is enough to show that $\mathcal{C}_q \otimes 1$ has compact resolvent, since $D^2$ and $\mathcal{C}_q \otimes 1$ differ by a bounded operator.
We have that $\mathcal{C}_q$ is central, hence it acts as a multiple of the identity in each irreducible representation.
These values are given in \cite[Proposition 3.3]{dd-proj}.
Since these values grow exponentially we conclude that $\mathcal{C}_q \otimes 1$ has compact resolvent.
\end{proof}

\begin{remark}
By analyzing the growth of the eigenvalues of $\mathcal{C}_q$ it can be seen that the spectral dimension of the corresponding spectral triple is zero, see \cite[Theorem 6.2]{dd-proj}.
On the other hand it is shown in \cite{mat14a} that, by including the modular operator corresponding to the Haar state, one recovers the classical dimension.
See also \cite{mat14b} for a discussion of the more general case of quantized irreducible flag manifolds.
\end{remark}

\appendix

\section{The Casimir element}
\label{sec:casimir}

In this appendix we give some details regarding the element $C = \sum_{i = 1}^{N} q^{-i}\mathcal{E}_{i} \mathcal{E}_{i}^{*}$, which appears in the computation of the square of the Dolbeault–Dirac operator.
It is easy to show, given the previous results, that it commutes with all elements of $U_{q}(\mathfrak{l})$.
On the other hand obtaining a central element in $U_{q}(\mathfrak{g})$ of a similar form requires more work.
For this we will refer to the detailed computation which is given in \cite{dd-proj}.

\begin{proposition}
The element $C = \sum_{i = 1}^N q^{-i} \mathcal{E}_i \mathcal{E}_i^*$ commutes with all the elements of $U_q(\mathfrak{l})$.
\end{proposition}

\begin{proof}
Consider the element $\tilde{C} = \sum_{i = 1}^N q^{-3i} E_{\xi_i}^* E_{\xi_i}$. It clearly commutes with the generators $K_i$ for $i = 1, \cdots, N$.
Now we check that it commutes with $E_j$ for $j = 1,\cdots, N - 1$.
Recall that for a Hopf algebra the condition of $X$ being central is equivalent to $X \triangleright Y = \varepsilon(X) Y$
for all $Y$. Using $X \triangleright (Y Z) = (X_{(1)} \triangleright Y) (X_{(2)}\triangleright Z)$ we get
\[
E_{j} \triangleright (E_{\xi_{i}}^{*} E_{\xi_{i}}) = (E_{j} \triangleright E_{\xi_{i}}^{*}) E_{\xi_{i}} + (K_{j} \triangleright E_{\xi_{i}}^{*}) (E_{j} \triangleright E_{\xi_{i}}).
\]
Using the relation $X \triangleright Y^* = (S(X)^* \triangleright Y)^*$ we can rewrite this expression as
\[
\begin{split}
E_{j} \triangleright (E_{\xi_{i}}^{*} E_{\xi_{i}})
& = -(K_{j} F_{j} K_{j}^{-1} \triangleright E_{\xi_{i}})^{*} E_{\xi_{i}} + (K_{j}^{-1}\triangleright E_{\xi_{i}})^{*} (E_{j} \triangleright E_{\xi_{i}})\\
& = -q^{-2} (F_{j} \triangleright E_{\xi_{i}})^{*} E_{\xi_{i}} + (K_{j}^{-1} \triangleright E_{\xi_{i}})^{*} (E_{j}\triangleright E_{\xi_{i}}).
\end{split}
\]
From \cref{lem:adj-act} we have $E_{j} \triangleright E_{\xi_{i}} = -\delta_{j, i - 1} E_{\xi_{i - 1}}$
and $F_{j} \triangleright E_{\xi_{i}} = -\delta_{j, i} E_{\xi_{i + 1}}$. Then
\[
E_{j} \triangleright (E_{\xi_{i}}^{*} E_{\xi_{i}}) = \delta_{j, i} q^{-2} E_{\xi_{i + 1}}^{*} E_{\xi_{i}} - \delta_{j,i - 1} (K_{j}^{-1} \triangleright E_{\xi_{i}})^{*} E_{\xi_{i-1}}.
\]
We have the identity $K_{i - 1}^{-1} \triangleright E_{\xi_{i}} = q E_{\xi_{i}}$, since this is true $i = N$ and for $i < N$ we have $E_{\xi_{i}} = -E_{i} E_{\xi_{i + 1}} + q^{-1} E_{\xi_{i + 1}}E_{i}$ from \cref{lem:recursion-roots}.
Therefore we obtain
\[
E_{j} \triangleright (E_{\xi_{i}}^{*} E_{\xi_{i}}) = \delta_{j, i} q^{-2} E_{\xi_{i + 1}}^{*} E_{\xi_{i}} - \delta_{j,i - 1} q E_{\xi_{i}}^{*} E_{\xi_{i - 1}}.
\]
Finally we use this result to compute
\[
E_{j} \triangleright \tilde{C}
= \sum_{i = 1}^{N} q^{-3i} E_{j} \triangleright (E_{\xi_{i}}^{*} E_{\xi_{i}})
= q^{-3j} q^{-2} E_{\xi_{j + 1}}^{*} E_{\xi_{j}} - q^{-3(j + 1)} q E_{\xi_{j + 1}}^{*} E_{\xi_{j}} = 0.
\]
The previous computation also implies that $\tilde{C}$ commutes with $F_{j}$ for $j = 1,\cdots, N - 1$. Indeed using the fact that $\tilde{C}^{*} = \tilde{C}$ we have
\[
0 = [E_{j}, \tilde{C}]^{*} = [\tilde{C}^{*}, E_{j}^{*}] = [\tilde{C}, K_{j} F_{j}] = K_{j} [\tilde{C}, F_{j}],
\]
where in the last step we have used the fact that $K_{j}$ commutes with $\tilde{C}$.

Finally to obtain $C$ consider $S^{-1}(\tilde{C})$. This element still commutes with $U_{q}(\mathfrak{l})$. We have
\[
S^{-1}(\tilde{C}) = \sum_{i = 1}^{N} q^{-3i} S^{-1}(E_{\xi_{i}}) S^{-1}(E_{\xi_{i}}^{*}) = \sum_{i = 1}^{N} q^{-3i} S^{-1}(E_{\xi_{i}}) S(E_{\xi_{i}})^{*}.
\]
Now we use the relation $S(E_{\xi_{i}}) = q^{-2(N + 1 - i)} S^{-1}(E_{\xi_i})$ obtained previously.
Then we get $S^{-1}(\tilde{C}) = q^{-2(N + 1)} \sum_{i = 1}^{N} q^{-i} \mathcal{E}_{i} \mathcal{E}_{i}^{*}$, which is equal to $C$ up to an overall constant.
\end{proof}

More generally the previous proof shows that an element of the form $\sum_{i = 1}^N q^{-i}K_{\omega_N}^m \mathcal{E}_i \mathcal{E}_i^* K_{\omega_N}^n$ commutes with all the elements in $U_q(\mathfrak{l})$, since $K_{\omega_N}$ is central in $U_q(\mathfrak{l})$.
Here $\omega_N$ is one of the fundamental weights, which in the simply-laced case are defined by $(\omega_i, \alpha_j) = \delta_{i j}$.

Our goal is to obtain a central element in $U_q(\mathfrak{g})$.
Unfortunately it does not seem possible to proceed as above, since the adjoint action only takes a simple form for elements of $U_q(\mathfrak{l})$.
For this reason we will use some of the results which appear in \cite{dd-proj}.
We should point out that their conventions for $U_q(\mathfrak{g})$ are different from ours. We can obtain their relations in terms of our generators after making the replacements
\[
K_i \to K_i^{-1/2}, \quad E_i \to K_i^{1/2} F_i, \quad F_i \to E_i K_i^{-1/2}.
\]

We will consider the Casimir operator defined in the cited paper, which we denote by $\mathcal{C}_q$. The most useful expression for our purposes is given in \cite[Lemma 6.5]{dd-proj}, where we have $\mathcal{C}_{q} \sim \sum_{i = 1}^N q^{-2i} X_i X_i^*$, up to a rescaling and neglecting terms in $U_{q}(\mathfrak{l})$ as usual.
The elements $\{X_i\}_{i = 1}^N$ are defined in \cite[Lemma 3.13]{dd-proj} and $\ell = N$ in our notation.
It is easy to show by induction that 
$X_i = - q^{-(N - i + 2)/2} K_{\omega_N} S^{-1}(E_{\xi_i})$.
We should also point out that we have $\hat{K} = K_{\omega_N}^{-1}$ in the notation of the cited paper.
Therefore we obtain the identities $X_i = q^{i / 2} K_{\omega_N} \mathcal{E}_i$, up to an overall constant.
From this it follows that
\[
\mathcal{C}_{q} \sim \sum_{i = 1}^N q^{-i} K_{\omega_N} \mathcal{E}_i \mathcal{E}_i^* K_{\omega_N}
= \sum_{i = 1}^N q^{-i} \mathcal{E}_i \mathcal{E}_i^* K_{\omega_N}^2.
\]
The action of the central element $\mathcal{C}_q$ is computed in \cite[Proposition 3.3]{dd-proj}.
Moreover it is shown in the appendix of the cited paper that $\mathcal{C}_q$ reduces to the quadratic Casimir in the classical limit $q \to 1$.
Finally we would like to point that the Dolbeault–Dirac operator in \cite{dd-proj} is defined in terms of the elements $\hat{K} X_i$, hence its definition involves the quantum root vectors, even though this notion is never explicitly introduced in the paper.

\vspace{3mm}

{\footnotesize
\emph{Acknowledgements}. I would like to thank Ulrich Krähmer, Sergey Neshveyev and Lars Tuset for the discussions that led to this paper.
Similarly I thank Réamonn Ó Buachalla for his comments.
Finally I should thank Robert Yuncken for discussions which led to the correction of a mistake in the first version.

The first version of this paper was written during my stay at the University of Oslo, where I was supported by the European Research Council under the European Union's Seventh Framework Programme (FP/2007-2013) / ERC Grant Agreement no. 307663 (P.I.: S. Neshveyev).
}

\bigskip

\end{document}